\documentclass[12pt]{amsart}
\usepackage{graphicx} 

\usepackage[backend=bibtex,style=alphabetic]{biblatex}
\usepackage{amssymb}
\usepackage{amsthm}
\usepackage{url}
\usepackage{graphicx}
\usepackage{tikz-cd}
\usetikzlibrary{babel}
\usepackage{tikz}
\usepackage{xcolor}
\usetikzlibrary{matrix,arrows,decorations.pathmorphing}
\usepackage{mathtools}
\usepackage{mathrsfs}
\usepackage{comment}
\usepackage{esvect}
\usepackage{stmaryrd}
\usepackage[colorlinks = true,
linkcolor = blue,
urlcolor  = blue,
citecolor = blue,
anchorcolor = blue]{hyperref}
\usepackage[all]{xy}
\theoremstyle{definition}
\newtheorem{thm}{Theorem}[section]
\newtheorem*{theorem*}{Theorem}

\newtheorem{prop}[thm]{Proposition} 

\newtheorem{cor}[thm]{Corollary}
\newtheorem{lem}[thm]{Lemma}
\newtheorem{thm-defi}[thm]{Theorem-Definition}

\newtheorem{defi}[thm]{Definition}
\newtheorem{rem}[thm]{Remark}
\newtheorem{ex}[thm]{Example}

\allowdisplaybreaks

\newcommand{\FF}{\mathbb{F}}
\newcommand{\BB}{\mathbb{B}}
\newcommand{\EE}{\mathbb{E}}
\newcommand{\ZZ}{\mathbb{Z}}

\newcommand{\RR}{\mathbb{R}}
\newcommand{\NN}{\mathbb{N}}

\newcommand{\PP}{\mathbb{P}}

\newcommand{\cB}{\mathcal{B}}

\newcommand{\cT}{\mathcal{T}}

\newcommand{\Chi}{\mathcal{X}}
\DeclareMathOperator{\Stab}{Stab}
\DeclareMathOperator{\Aut}{Aut}

\DeclareMathOperator{\St}{St}

\DeclareMathOperator{\Ker}{Ker}

\DeclareMathOperator{\GL}{GL}
\DeclareMathOperator{\GQ}{GQ}

\DeclareMathOperator{\h}{H}

\DeclareMathOperator{\End}{End}

\DeclareMathOperator{\Hom}{Hom}

\DeclareMathOperator{\Spec}{Spec}

\DeclareMathOperator{\coker}{coker}

\DeclareMathOperator{\Id}{Id}

\DeclareMathOperator{\C}{C}
\DeclareMathOperator{\im}{Im}

\newcommand{\st}{\mathrm{st}}
\newcommand{\un}{\mathrm{un}}
\newcommand{\dashun}{\textrm{-}\un}
\newcommand{\dashst}{\textrm{-}\st}

\begin{document}
	
	\title{The Unstable Complex in Bruhat-Tits Buildings\\ for Arithmetic Groups over Function Fields}
	\author{Gebhard B\"{o}ckle}
\address{Heidelberg University, IWR, Im Neuenheimer Feld 205, 69120, Heidelberg, Germany}
\email{gebhard.boeckle@iwr.uni-heidelberg.de}

\author{Sriram Chinthalagiri Venkata}
\address{Heidelberg University, IWR, Im Neuenheimer Feld 205, 69120, Heidelberg, Germany}
\email{sriram.chinthalagiri@iwr.uni-heidelberg.de}

	\date{\today}

\let\origmaketitle\maketitle
\def\maketitle{
  \begingroup
  \def\uppercasenonmath##1{} 
  \let\MakeUppercase\relax 
  \origmaketitle
  \endgroup
}

	\begin{abstract}
		Let $K$ be a function field in positive characteristic, $\infty$ be a fixed place of $K$ and $K_\infty$ be the completion of $K$ at $\infty$. By the work of Serre, it is well known that, for a suitable arithmetic subgroup $\Gamma \subset \GL_2(K)$, the $\Gamma$-unstable region of the Bruhat-Tits tree for $\GL_2(K_\infty)$ is naturally homotopy equivalent to the spherical Tits building for $\GL_2(K)$. Grayson, following Quillen's ideas, generalizes this homotopy equivalence to the non-semistable part of the Bruhat-Tits building for $\GL_r(K_\infty)$. Modifying the approach described by Grayson, we are also able show a similar homotopy equivalence for the $\Gamma$-unstable region, for $\Gamma \subset \GL_r(K)$ a principal congruence subgroup.    
	\end{abstract}
	
	\maketitle

	\section{Introduction}
	
	Let $\mathbb{F}_q$ be the finite field of order $q$ and prime characteristic $p$, so that $q=p^n$ for some $n$. Let $C$ be a smooth, geometrically irreducible projective curve over $\mathbb{F}_q$ with function field $K:=K(C)$. Fix a closed point $\infty$ of $C$ and let $A:= \Gamma (C - \infty,\mathcal{O}_{C})$ be the ring of functions on $C$ that are regular outside $\infty$, and $R\subset K$ the discrete valuation ring $\mathcal{O}_{C,\infty}$. We write $\varpi \in R$ for a uniformizer. The ring $R$ is not complete. We let $K_\infty$ be the completion of $K$ at $\infty$ and $\widehat R$ the completion of $R$ inside~$K_\infty$.
	
	Denote by  $W$ a $K$-vector space of finite dimension $r\ge2$. An interesting combinatorial object for actions of congruence subgroups of $\GL(W)$, such as $\GL_r(A)$, is the  \textit{Bruhat-Tits building of rank $r-1$} 
	that arises from the chosen place $\infty$; we shall denote it $\cB_\bullet(W)$, omitting the dependency of $\infty$. 
	Then $\cB_\bullet(W)$ is the simplicial complex whose vertices are homothety classes of full rank $R$-lattices inside $W$, and whose $d$-simplices, for an integer $d \geq 0$, are $d+1$-element sets of classes of $R$-lattices, such that for suitable lattice representatives one has $L_0\supsetneq L_1\supsetneq\ldots\supsetneq L_d\supsetneq \varpi L_0$. It turns out that $\cB_\bullet(W)$ is a chamber complex with chambers consisting of $(r-1)$ -simplices. The natural action of $\GL(W)$ on $\cB_\bullet(W)$ induces, for each $g \in \GL(W)$, a simplicial automorphism.

To motivate our work, let us briefly consider the case $r=2$, where $\cB_\bullet(W)$ is a Bruhat-Tits tree, and let $\Gamma\subset \GL(W)$ be a congruence subgroup such that $\Gamma$ contains no elements of finite order prime-to-$p$. This case is much studied in Serre's book \cite{serre}, and the following observations are from \cite[Section 2.9]{serre}. The group $\Gamma$ acts on $\cB_\bullet(W)$, and one calls a simplex $s$ of $\cB_\bullet(W)$ stable if its stabilizer $\Gamma_s=\Stab_\Gamma(s)$ is trivial and unstable otherwise. The unstable simplices form a sub simplicial complex $\cB_\bullet(W)^{\Gamma\dashun}$ of $\cB_\bullet(W)$. The following results are shown in \cite[Section 2.9]{serre}:
\begin{enumerate}
\item The complex  $\cB_\bullet(W)^{\Gamma\dashun}$ is a disjoint union of trees.
\item The trees in (1) can be canonically labeled by the elements in the $K$-rational projective space $\PP(W)$, that is the Tits building of $W$. 
\item This labelling in (2) can be interpreted as contracting the unstable complex to the boundary, or as a homotopy equivalence between $\cB_\bullet(W)^{\Gamma\dashun}$ and $\PP(W)$.
\item The quotient complex $\ZZ[\cB_\bullet(W)]^{\Gamma\dashst}:=\ZZ[\cB_\bullet(W)]/\ZZ[\cB_\bullet(W)^{\Gamma\dashun}]$ is a complex of finite free $\ZZ[\Gamma]$-modules concentrated in degrees $0$ and $1$.
\item Let $\St:=\ker(\ZZ[\PP(W)]\to\ZZ)$ be the Steinberg module; it is a $\ZZ[\GL(W)]$-module. Then the complex in (4) is a resolution of $\St$ as a $\ZZ[\Gamma]$-module, so that $\St$ is a finitely generated projective $\ZZ[\Gamma]$-module.
\item As a consequence, if we pass from $\Gamma$ to a congruence subgroup $\Gamma'\subset\Gamma$ that is normal in $\Gamma$, then the unstable trees in (1) shrink, but do not decompose. Moreover the canonical surjection $\ZZ[\cB_\bullet(W)]^{\Gamma'\dashst} \twoheadrightarrow \ZZ[\cB_\bullet(W)]^{\Gamma\dashst}$ is a homotopy equivalence of complexes of $\ZZ[\Gamma]$-modules.
\end{enumerate}
	
%
%
%
The aim of the paper is to obtain comparable results for $r>2$, where methodically we adapt the approach of Quillen from \cite{grayson}. Obviously (1) and (2) cannot hold on the nose, as for $r>2$ the unstable complex is connected. But a homotopy equivalence as in (3) from a Hardar-Narasimhan unstable complex to the Tits-building had been given in \cite{grayson}. Using this notion of stability, \cite{grayson} also establishes variants of (4) and (5); see \cite[Cor.4.2]{grayson}. The use of Harder-Narasimhan semistability is possible because \cite{grayson} interprets the vertices of $\cB_\bullet(W)$ as rank $r$ vector bundles on $C$; see \S3.2.


So what we do is to replace Hardar-Narasimhan stability by a stability notion built on sufficiently small congruence subgroups $\Gamma$ of $\GL(W)$ and establish natural analogs of (3)--(6). This necessitates a major adaption of the methods from \cite{grayson} to our setting.
	
	To state our main results, let us introduce  \textit{principal congruence subgroups} inside $\GL(W)$ that are the arithmetic groups $\Gamma$ that we consider. They depend on a projective $A$-submodule $P \subset W$ of rank $r$ and a proper non-zero ideal $I \subset A$, and are defined as the kernel of the natural map $\Aut_A(P) \rightarrow \Aut_A(P/IP)$, and denoted $\Gamma_P(I)$. For simplicity, in this introduction, we abbreviate $\Gamma:=\Gamma_P(I)$.

%
	 
The simplicial complex $\cB_\bullet(W)$ contains as a natural subcomplex, the complex $\cB_\bullet(W)^{\Gamma\dashun}$ of $\Gamma$-unstable simplices, i.e., of simplices whose $\Gamma$-stabilizer is non-trivial. For a decreasing chain of ideals, the resulting subcomplexes seem to shrink towards a boundary of the building, and our main result gives some precise meaning to this. Let  $\mathcal{T}_\bullet(W)$ be the spherical Tits building for $\GL(W)$, which may be regarded as the $K$-rational boundary of  $\cB_\bullet(W)$. Then we prove the following theorem.
	\begin{theorem*}(see Theorem \ref{T:main_thm_2})
	There is a 
	 natural $\Aut_A(P)$-equivariant homotopy equivalence between $\cB_\bullet(W)^{\Gamma\dashun}$ and $\mathcal{T}_\bullet(W)$.
	\end{theorem*}

The formulation and proof of the above theorem is modelled along that of \cite{grayson}. In particular, the idea is to find suitable subcomplexes $\cB_\bullet(W)_\sigma \subset \cB_\bullet(W)^{\Gamma\dashun}$, for any vertex $\sigma \in \mathcal{T}_0(W)$, which are contractible and have the following further property. If for any simplex $\gamma \in \cB_\bullet(W)^{\Gamma\dashun}$, we gather the various $\sigma$'s in $\mathcal{T}_\bullet(W)$ so that $\gamma \in \cB_\bullet(W)_\sigma$ into $\mathcal{T}_\bullet(W)_\gamma$, then 
also $\mathcal{T}_\bullet(W)_\gamma$ is a contractible subcomplex of $\mathcal{T}_\bullet(W)$; see Theorem \ref{T:lem1.9} for details, which is an equivariant version of \cite[Lem. 1.9]{grayson}.
It is worth pointing out here that, unlike \cite{grayson} where the corresponding $\cB_\bullet(W)_\sigma$'s are defined for simplices $\sigma \in \mathcal{T}_\bullet(W)$ thereby apriori showing a homotopy equivalence to the barycentric subdivision of $\mathcal{T}_\bullet(W)$, we take a cue from Serre's treatment in rank 2 and instead directly show the homotopy equivalence to the Tits building. 

In the concluding section, we explicitly describe a (simplicial) homological chain complex, $C_\bullet(\cB_\bullet(W),\ZZ)$ and consequently also $C_\bullet(\cB_\bullet(W)^{\Gamma\dashun},\ZZ)$. Then we define $\St^I(P)$ to be the $(r-1)$-st homology of the cokernel of $C_\bullet(\cB_\bullet(W)^{\Gamma\dashun},\ZZ)  \hookrightarrow C_\bullet(\cB_\bullet(W),\ZZ)$, where the cokernel corresponds to a complex consisting of $\Gamma$-stable chains. Using the theorem above as well as the long exact sequence in homology, there is a natural $\Gamma_P$-equivariant isomorphism $\GQ_I:\St^I(P) \simeq \St(W)$. Moreover in Theorem~\ref{thm:Thm2} we show that these isomorphisms are compatible as $I$ runs through non-zero proper ideals of $A$. 

The results of this paper are part of the PhD work of the second author under the guidance of the first author. In this ongoing doctoral work we use these results to better understand the relation of the harmonic cochains on the Bruhat-Tits building with the Steinberg module in higher ranks. In particular, for an ordered $K$-basis $\underline{w}=(w_1,\ldots,w_r)$, we expect to relate the cycle of $\Gamma$-stable chambers in the apartment associated with $\underline{w}$ to the modular symbol $[[\underline{w}]]$ of Ash-Rudolph(\cite{Ash_1}) via the map $\GQ_I$. Using this we aim to generalize certain results of Teitelbaum in rank $2$, regarding the existence and uniqueness of an extension of a harmonic cochain on the $\Gamma$-stable part to a  $\Gamma$-invariant harmonic cochain on the full building $\cB_\bullet(W)$; cf.~\cite[Lemma 19]{teitelbaum}.  

\subsection{Outline of the paper.} In \S2 we collect the necessary definitions and results regarding simplicial complexes. Since we are concerned with equivariant homotopy equivalences, following \cite{thevenaz} we record in this section equivariant counterparts of some results in \cite{grayson}. \S3 is devoted to the definition(s) of the Bruhat-Tits building as well as the action of an arithmetic subgroup on it. In \S4, we formulate and prove the main result of this paper. We conclude with \S5 on the relation between the various unstable complexes obtained by considering the arithmetic subgroups $\Gamma_P(J)$ for $J \subset I$ any non-zero proper ideal of $A$.

\subsection*{Acknowledgements}

Both the authors acknowledge support by Deutsche Forschungsgemeinschaft (DFG) through CRC-TRR 326 `Geometry and Arithmetic of Uniformized Structures', project number 444845124.

	\section{Preliminaries on simplicial complexes}
	
	We collect here important definitions and results on simplicial complexes which will be required by us later. We mainly follow \cite[\S1]{grayson} and for proving equivariant versions of Grayson's results we follow \cite{thevenaz}. For more details, we refer the reader there.

	Recall that a \textit{simplicial complex} is a pair $(V,X)$, where $V$ is called the set of vertices and $X$ is a collection of finite subsets of $V$ such that if $Y \in X$ and $Z \subset Y$, then $Z \in X$. A subset $Y \in X$ of $V$ of cardinality $d+1$, for $d \geq 0$, is called a \textit{$d$-simplex.} In particular, vertices are \textit{$0$-simplices}. We denote the collection of all $d$-simplices of the pair $(V,X)$ by $X_d$. Since $X_0=V$, we will often denote the pair $(V,X)$ by $X$. A simplicial complex $X$ is said to be a rank $n\in\NN_0$, if $X_n \neq \emptyset$ and $X_i = \emptyset  \ \forall i > n$. 
A \textit{simplicial map} $f:X \rightarrow Y$ between simplicial complexes is a map $X_0 \rightarrow Y_0$ such that the image of a simplex is again a simplex.   
	
	By the axiom of choice, we can choose a partial ordering on $X_0$ so that each simplex becomes totally ordered. This motivates the following definition
	
	\begin{defi}
		An \textbf{ordered simplicial complex} is a simplicial complex $X$ together with a partial order $\leq$ on $X_0$, which makes each simplex totally ordered. 
	\end{defi}
	
	\begin{ex}
		Let $d\geq 0$ be an integer. The \textit{standard $d$-simplex}, denoted $\Delta_d$, is the ordered simplicial complex with vertices $\{0,\ldots,d\}$, together with the order inherited from $\NN_0$, and its simplices are all finite subsets of $\{0,\ldots,d\}$. It has rank $d$.
	\end{ex}
	\subsection{\textnormal{\emph{Geometric realization, ($G$-)CW complexes and equivariant Whitehead's theorem}}}
	
	Let $
	[0,1]$ be the closed unit interval in $\RR$.
	One can attach to any simplicial complex $X$ its topological realization $|X|$, that may be described as a topological subspace of $[0,1]^{X_0}$, by
 \[|X| := \Big\{(t_v) \in \underset{v \in X_0}{\prod} [0,1] \mid \{v \in X_0|t_v \neq 0\} \text{ is a simplex in $X$ and } \underset{v \in X_0}{\sum} t_v=1\Big\}.\] In particular, we see that $|\Delta_d| =\{(t_i)^{d-1}_{i=0} \in \RR^{d+1}| \sum_i t_i =1\}$, which is the standard $d$- geometric simplex. Consequently, for each simplex $\sigma \in X_d$ which is simplicial complex in its own right, $|\sigma|$ is homeomorphic to $|\Delta_d|$. So now we can define the topology on $|X|$ as the weak topology with respect to these geometric simplices, in the sense that $M \subset |X|$ is closed if and only if $M \cap |\sigma| \subset |\sigma|$ is closed for all $\sigma \in X$.  
	
	Let $G$ be a group. A \textit{$G$-simplicial complex} is a simplicial complex $X$ together with an action of $G$ on $X$ such that each $g \in G$ acts via a simplicial map on $X$. Clear one has an induced $G$ action on $|X|$, and for it, the action of any $g\in G$ is continuous. Later we will regard the geometric realization $|X|$ as a $G$-CW complex. We recall the required definition. Call a topological space $\mathcal{Y}$ a $G$-space, if any $g\in G$ acts left continuously on $\mathcal{Y}$.  
	
	\begin{defi}(\cite[\S2]{mislin}, \cite[Ch.4, \S38]{Mun84}, \cite{may})\label{D:cwcomplex}
		\begin{itemize}
			\item[(i)] A topological space $\mathcal{Y}$ is said to be a \textbf{CW complex of rank $n$} if it admits a filtration by closed subsets $\mathcal{Y}_0 \subset \mathcal{Y}_1 \subset \ldots \subset \mathcal{Y}_n=\mathcal{Y}$ such that the following are satisfied:
			\begin{itemize}
				\item[(a)] $\mathcal{Y}_0 \subset \mathcal{Y}$ is a discrete topological space;
				\item[(b)] Let $\mathbb{D}^d:=\{(x_i)^{d}_{i=1} \in \RR^d|\ \sum_i x^2_i\leq 1\}$ denote the $d$-dimensional unit disc and $\mathbb{S}^{d-1}:=\partial \mathbb{D}^d$  the corresponding $(d-1)$-dimensional unit sphere. Then there exists a discrete topological space $P_d$ for all $0 \leq d \leq n$ and maps $f_d:\mathbb{S}^{d-1} \times P_d \rightarrow \mathcal{Y}_{d-1}$ and $\widehat{f}_d:\mathbb{D}^{d} \times P_d \rightarrow \mathcal{Y}_{d}$ such that we have a pushout diagram \[\begin{tikzcd}
					\mathbb{S}^{d-1} \times P_d \arrow[r,"f_d"] \arrow[d]
					& \mathcal{Y}^{d-1} \arrow[d]\\
					\mathbb{D}^{d} \times P_d \arrow[r,"\widehat{f}_d"] 
					& \mathcal{Y}_{d} \arrow[ul, phantom, "\lrcorner", very near start]
				\end{tikzcd}\] where the left vertical arrow is given by natural inclusion on the first coordinate and identity on the second coordinate, and the right vertical arrow is inclusion. 
			\end{itemize}
			\item[(ii)] Moreover, let $\mathcal{Y}$ be a $G$-space. Then $\mathcal{Y}$ is said to be a \textbf{G-CW complex}, if $\mathcal{Y}$ is a CW complex as above,  with the further condition that each $\mathcal{Y}_i$ is a closed $G$-subspace of $\mathcal{Y}$, each $P_d$ is a $G$-set and all the maps $f_d$ and $\widehat{f_d}$ are $G$-equivariant(with $G$-acting trivially on $\mathbb{S}^{d-1}$ and $\mathbb{D}^d$). 
		\end{itemize}
		
	\end{defi}

	\begin{lem}\label{L:simpl_cw}
		Let $X$ be a $G$-simplicial complex of rank $n$. Then $|X|$ is a $G$-CW complex. Moreover, for any two $G$-simplicial complexes $X,Y$, any $G$-simplicial map $X \rightarrow Y$ induces a canonical continuous $G$-map of $G$-CW complexes $|X| \rightarrow |Y|$.   
	\end{lem}
	
	\begin{proof}
		We first consider the case $G=\{1\}$. In the notation of the definition above, we set $\mathcal{Y}=|X|$, $\mathcal{Y}_d=\underset{\sigma \in X_d}{\bigcup} |\sigma|$ and $P_d=X_d$. We claim that $|X|$ becomes a CW complex with respect to these $\mathcal{Y}_i$'s and $P_d$'s. Condition \ref{D:cwcomplex}(i)(a) is clearly true since $|X|$ is equipped with a weak topology. 
		For a simplex $\sigma \in X_d$, denote by $\partial \sigma:=\{\tau| \tau \subset \sigma \land \tau \in X_{d-1}\}$. We choose a pushout diagram 
		\[\begin{tikzcd}
			\mathbb{S}^{d-1} \times \{\sigma\} \arrow[r,"f_{d}"] \arrow[d]
			& \cup_{\tau \in \partial\Delta^d} |\tau| \arrow[d]\\
			\mathbb{D}^{d} \times \{\sigma\} \arrow[r,"\widehat{f}_{d}"] 
			& \text{\textbar} \Delta^d|
			\arrow[ul, phantom, "\ulcorner", very near start] 
		\end{tikzcd}\] 		
		where $f_{d}$ and $\widehat{f}_{d}$ are suitable homeomorphisms. One can use this model diagram for any $\sigma\in X_d$ in place of $\Delta^d$, say with homeomorphisms $f_{d,\sigma}$ and $\widehat{f}_{d,\sigma}$, respectively. 
		Considering $\cup_{\sigma \in X_d} f_{d,\sigma}$ and $\cup_{\sigma \in X_d} \widehat{f}_{d,\sigma}$, we see that also Condition \ref{D:cwcomplex}(i)(b) holds for $|X|$, and hence $|X|$ is a CW complex.

		It is clear from the construction of $\mathcal{Y}_i$'s, $P_d$'s and $f_d,\widehat{f_d}$, that all the necessary conditions of \ref{D:cwcomplex}(ii) are also satisfied by them, so that $|X|$ is a $G$-CW complex, concluding the proof. 
	\end{proof}
	
	$G$-CW complexes have very nice $G$-equivariant homotopy properties. A fundamental and very useful result is the equivariant version of Whitehead's theorem. Before stating it, let us recall the following definitions.
	\begin{defi}\label{D:homotopy_equi}
		\begin{itemize}
			\item[(i)] Let $\mathcal{Y},\mathcal{Z}$ be two $G$-spaces and $f,g:\mathcal{Y} \rightarrow \mathcal{Z}$ be two continuous maps of $G$-spaces. Then $f$ and $g$ are said to be \textbf{$G$-homotopic} to each other, if there is a homotopy $F:\mathcal{Y} \times [0,1] \rightarrow \mathcal{Z}$ between $f$ and $g$ such that $F$ is also $G$-equivariant(with $G$ acting trivially on $[0,1]$). 
			\item[(ii)] A $G$-map $f:\mathcal{Y} \rightarrow \mathcal{Z}$ which is a homotopy equivalence is said to be a \textbf{$G$-equivariant homotopy equivalence}. If in addition  there exists a $G$-map $f':\mathcal{Z} \rightarrow \mathcal{Y}$ such that $f' \circ f$ is $G$-homotopic to $\Id_\mathcal{Y}$ and $f \circ f'$ is $G$-homotopic to $\Id_\mathcal{Z}$, then $f$ is said to be a \textbf{$G$-homotopy equivalence}. 
			\item[(iii)] A $G$-space $\mathcal{Y}$ is said to be \textbf{$G$-contractible}, if $f:\mathcal{Y} \rightarrow \{\ast\}$ is a $G$-homotopy equivalence.  
		\end{itemize}
	\end{defi}
	
	Note that below, for a $G$-space $\mathcal{Y}$ and a subgroup $H < G$, we set \[\mathcal{Y}^H:=\{y \in \mathcal{Y}|\ hy=y \ \forall \ h \in H\}.\]
	
	\begin{rem}
		A necessary condition for a $G$-space $\mathcal{Y}$ to be $G$-contractible is that $\mathcal{Y}^G \neq \emptyset$. Indeed, from (ii) above, we see that the existence of a $G$-map $f':\{\ast\} \rightarrow \mathcal{Y}$ implies that the image of $\ast$ is a $G$-invariant point. 
	\end{rem}

	\begin{thm}(Equivariant Whitehead's theorem I, \cite[Thm.2.2]{mislin}, \cite[Cor.3.3]{may})\label{T:equi_whitehead}
		Let $\mathcal{Y},\mathcal{Z}$ be $G$-CW complexes and $f:\mathcal{Y} \rightarrow \mathcal{Z}$ a $G$-map between them. If for all subgroups $H < G$, for all $y_0 \in \mathcal{Y}^H$ and for all $i\geq 0$, the induced map on homotopy groups $\pi_i(\mathcal{Y}^H,y_0) \rightarrow \pi_i(\mathcal{Z}^H,f(y_0))$ is an isomorphism, then $f$ is a $G$-homotopy equivalence.  
	\end{thm}
	
	Applying the theorem to the constant map $\mathcal{Y} \rightarrow \{\ast\}$, we get as an immediate consequence,  
	\begin{cor}(Equivariant Whitehead's theorem II)\label{C:cor_equi_white}
		Let $\mathcal{Y}$ be a $G$-space and suppose that for all subgroups $H < G$, $\mathcal{Y}^H \neq \emptyset$. If for any choice of base points $y_0 \in \mathcal{Y}^H$ and all $i \geq 0$, the homotopy groups $\pi_i(\mathcal{Y}^H,y_0)$ are trivial, then $\mathcal{Y}$ is $G$-contractible. 
	\end{cor}

\begin{rem}\label{R:non-equi-whitehead}
If $G$ is the trivial group, then the above results are simply the standard, non-equivariant forms of Whitehead's Theorem.
\end{rem}	

	\subsection{\textnormal{\emph{Homotopy properties of $G$-simplicial complexes}}}
	
	\begin{defi}(\cite[Defn.1.2]{grayson})
		Let $X,Y$ be $G$-simplicial complexes with $X$ ordered. Two $G$-simplicial maps $f,g:X \rightarrow Y$ are said to be \textbf{adjacent}, if for each simplex $\sigma$ of $X$ and for each $x \in \sigma$, the set 
		\[\{f(x')|\ x' \in \sigma \text{ and } x' \leq x\} \cup \{g(x') | \ x' \in \sigma \text{ and } x \leq x'\}\] is a simplex in $Y$. 
	\end{defi}

	\begin{rem}
		
		The above definition corrects a typographical error of \cite[Definition~1.2]{grayson}. The correction is needed in the proof of Proposition \ref{P:adj_1} below that is \cite[Corollary~1.3]{grayson}.

	\end{rem}
	
	\begin{defi}\label{D:adj_1}
		Given ordered $G$-simplicial complexes $X,Y$, the \textbf{product simplical complex}, denoted $X \times Y$, is the simplical complex, whose vertex set is $X_0 \times Y_0$, and $\{(x_1,y_1),\ldots,(x_i,y_i)\}$ is a simplex iff $x_1 \leq \ldots\leq x_i$ (resp. $y_1 \leq \ldots\leq y_i$) and $\{x_1,\ldots,x_i\}$ (resp. $\{y_1,..,y_i\}$) is a simplex in $X$ (resp. $Y$). It is again a $G$-simplicial complex.
		
	\end{defi}
	
	\begin{lem}(\cite[Lemma 1.1]{grayson})\label{L:prod_sim}
		The natural map $|X \times Y| \rightarrow |X| \times |Y|$ is a $G$-homeomorphism. 
	\end{lem}
	
	\begin{proof}
		The proof of \cite[Lemma 1.1]{grayson} for $G=\{1\}$ extends to the case of general $G$.
	\end{proof}
	
	The notion of adjacent maps captures the idea of homotopy and on the level of simplicial complexes, it is explained by the following
	
	\begin{prop}\label{P:adj_1}
		Suppose $X$ is an ordered $G$-simplicial complex and $Y$ is a $G$-simplicial complex. Then adjacent simplicial maps from $X$ to $Y$ are $G$-homotopic, that is their geometric realizations are $G$-homotopic. 
	\end{prop}
	\begin{proof}
		Let $f,g:X \rightarrow Y$ be two adjacent $G$-simplicial maps. Recall that $\Delta_1=\{0,1\}$ is the $1$-simplex, with vertices being $0$ and $1$, and the only $1$-simplex being ${0,1}$ itself. Adjacency is precisely the condition required to ensure that the following map is $G$-simplicial(with $G$ acting trivially on $\Delta_1$)  \[F:X \times \Delta_1 \rightarrow Y \] sending \[(x,i) \mapsto \begin{cases}
			f(x) & \text{; }i=0\\
			g(x) & \text{; }i=1.
		\end{cases}.\] Lemma \ref{L:prod_sim} finishes the proof since the geometric realization $|\Delta_1|$ is the closed unit interval $[0,1]$. 
	\end{proof}
	

\subsection{\textnormal{\emph{Simplicial complexes on $G$-posets and their (equivariant) homotopy properties}}}
	
	Let $G$ be a group. In the following, by a \textit{$G$-poset} we mean a poset equipped with a (left)$G$-action which is compatible with the underlying partial order. Many results below are generalized from their non-equivariant counterparts, which can be recovered by simply choosing $G=1$(or by choosing a trivial action of $G$).
	
	\begin{ex}\label{E:poset-simplicial}
		\begin{enumerate}
			\item Let $\mathcal{T}$ be a $G$-poset. Then one obtains a $G$-simplicial complex, whose vertex set is $\mathcal{T}$ and whose simplices are the finite, totally ordered subsets of~$\cT$.
			\item Let $\EE$ be a $G$-poset and suppose that it carries a $\mathbb{Z}$ action commuting with the $G$-action. Suppose the $\mathbb{Z}$-action is \textbf{cofinal}, which means that for any $x,x' \in \EE$, there exists $n \in \mathbb{Z}$ s.t. $x+n \geq x'$. Then 
we denote by $\langle \EE \rangle$, the $G$-simplicial complex, whose vertex set are the $\mathbb{Z}$-orbits of $\EE$ and where a simplex is any finite nonempty set of vertices whose union is a chain in $\EE$. For $t \in \EE$, let $\langle t \rangle \in \langle \EE \rangle$ denote the $\ZZ$-orbit of~$t$. 
			\item The set of integers $\ZZ$ can be regarded as a $G$-poset, with trivial $G$-action and cofinal $\ZZ$-action. 
			\end{enumerate}
	\end{ex}
	
	\begin{lem}\label{L:poset_sim}(\cite[Lemma 1.5]{grayson})
		Let $\EE_1,\EE_2$ be two $G$-posets with a cofinal $\ZZ$-action 
		and let $f:\EE_1 \rightarrow \EE_2$ be a $G$-map satisfying 
		\begin{itemize}
			\item[(i)] for $t \leq t'$ in $\EE_1$, $f(t) \leq f(t')$;
			\item[(ii)] for all $t \in \EE
			_1$, $f(t+1)=f(t) + 1$.
		\end{itemize}
		Then the map $\langle f \rangle:\langle \EE_1 \rangle \rightarrow \langle \EE_2\rangle$, defined by sending $\langle x \rangle \mapsto \langle f(x) \rangle$, is a $G$-simplicial map. 
	\end{lem}
	
	\begin{proof}
		For $G=1$ this is proved in \cite[Lemma 1.5]{grayson}. 
		For general $G$ the proof is identical. 
	\end{proof}
	We would like to extend the Definition \ref{D:adj_1} to the situation of posets with a cofinal $\ZZ$-action, so as to obtain a result similar to Proposition \ref{P:adj_1}. For this, we need the following definition
	
	\begin{defi}\label{D:augmentation}
		Let $\EE$ be a $G$-poset with a cofinal $\ZZ$-action. A map $\varepsilon:\EE \rightarrow \ZZ$ is said to be an \textbf{augmentation map}, if $\varepsilon$ satisfies the two conditions of Lemma \ref{L:poset_sim}.
		

Given $(\EE,\varepsilon)$, let $\EE_0$ denote the ordered simplicial complex with vertex set $\{ t\in \EE\mid \varepsilon(f)=0\}$ and whose $q$-simplices are all sets $\{ t_0,\ldots,t_q\}$ of vertices which can be indexed so that $t_0<t_1<\ldots<t_q< t_0+1$.

	\end{defi}
	
	\begin{rem}\label{R:aug_map}
%
Given $(\EE,\varepsilon)$, 
the simplicial complex $\EE_0$ is isomorphic to $\langle \EE \rangle$, but $\EE_0$, being a subset of $\EE$,  is moreover also ordered.
		
	\end{rem}
	
	\begin{defi}\label{D:adj_2}
		Let $\EE_1,\EE_2$ be two $G$-posets with a cofinal $\ZZ$-action and $\varepsilon:\EE_1 \rightarrow \ZZ$ be an augmentation map. Then two $G$-maps $f,g:\EE_1 \rightarrow \EE_2$ are said to be \textbf{adjacent} if they satisfy the following:-
		\begin{itemize}
			\item[(i)] for $t \in \EE_1$, $f(t) \leq g(t)$;
			\item[(ii)] for $t < t'$ in $\EE_1$ with $\varepsilon(t) < \varepsilon(t')$, $g(t) \leq f(t')$.
		\end{itemize}
	\end{defi}
	
	We can now state the analogue of Proposition \ref{P:adj_1} in this situation
	
	\begin{prop}(\cite[Lemma 1.8]{grayson})
		\label{P:adj_2}
		With $\EE_1,\EE_2$ as above and letting $f,g:\EE_1 \rightarrow \EE_2$ be two adjacent $G$-maps, the geometric realizations of $\langle f \rangle,\langle g\rangle:\langle \EE_1\rangle \rightarrow \langle \EE_2\rangle$ are $G$-homotopic to each other.     
	\end{prop}
	
	\subsection{\textnormal{\emph{$G$-equivariant properties}}}
	
	We conclude this section by proving a $G$-equivariant version of \cite[Lemma 1.9]{grayson}. This requires us to also restate \cite[Cor.1.8]{quillen} in terms of $G$-equivariant homotopy, which we do so following \cite[Thm.1]{thevenaz}. 
	
	Checking homotopy properties of maps between posets can be easy. For example, if $f,g:\mathcal{U} \rightarrow \mathcal{T}$ is a map of posets such that $f(u) \leq g(u), \ \forall u \in \mathcal{U} $, then $|f|$ and $|g|$ are homotopic to each other(\cite[\S1.3]{quillen}). This can be enhanced to a $G$-homotopy, in the case of $G$-posets.
	\begin{lem}(\cite[Prop.1.1]{thevenaz})\label{L:quillen1.3}
		Let $\mathcal{U} , \mathcal{T}$ be $G$-posets and $f,g:\mathcal{U} \rightarrow \mathcal{T}$ two maps of $G$-posets such that $f(u) \leq g(u), \ \forall u \in \mathcal{U} $, then $|f|$ and $|g|$ are $G$-homotopic to each other.
	\end{lem}
	\begin{proof}
		Note that for each simplex $\sigma$ of $\mathcal{U}$ and for each $u \in \sigma$, the set 
		\[\{f(u')|\ u' \in \sigma \ \text{and} \ u' \leq u\} \cup \{g(u') | \ u' \in \sigma \ \text{and} \ u \leq u'\}\] is a totally ordered subset of $\mathcal{T}$ by hypothesis and hence a simplex in $\mathcal{T}$. Therefore $f,g$ are adjacent maps(\ref{D:adj_1}) and hence one can conclude by applying Proposition \ref{P:adj_1}. 
	\end{proof}
	We record two consequences of this lemma
	
	\begin{cor}\label{C:quillen1.3_1}
		Let $\mathcal{T}$ be a $G$-poset and suppose that it contains a smallest(or largest) element, which is moreover fixed under $G$. Then $\mathcal{T}$ is $G$-contractible. 
	\end{cor}
	\begin{proof}
	Let $x_0 \in \mathcal{T}$ denote the smallest element, which is also $G$-invariant. Define maps $f,g:\cT\to \cT$ by $f(x)=x_0$ and $g(x)=x$ for all $x\in\cT$.	
	By Lemma \ref{L:quillen1.3}, $g=Id_\mathcal{T}$ is $G$-homotopic to the constant map $f$, and hence $\mathcal{T}$ is $G$-contractible. The proof for $x_0$ being the largest element is also done similarly. 
	\end{proof}
	
	\begin{cor}(\cite[Cor.1.2]{thevenaz})\label{C:quillen1.3_2}
		Let $\mathcal{T}$ be a $G$-poset together with a map $\phi:\mathcal{T} \rightarrow \mathcal{T}$ such that $\phi(x) \geq x$ for all $x \in \mathcal{T}$. The $\phi:\mathcal{T} \rightarrow \im(\phi)$ is a $G$-homotopy equivalence. 
	\end{cor}
	
	Given a $G$-poset $\mathcal{T}$, for $\sigma \in \mathcal{T}$, we denote by $\mathcal{T}_{\leq \sigma}$(resp. $\mathcal{T}_{\geq \sigma}$) the set \[\{t \in \mathcal{T}|t \leq \sigma\}(\text{resp. }  \{t \in \mathcal{T}|t \geq \sigma\}).\] 
	For any $G$-set $X$ and $x \in X$, the stabilizer subgroup $G_x \subset G$ is defined as the subgroup of all $g \in G$ such that $gx=x$. We have the following $G$-equivariant version of \cite[Prop.1.6]{quillen}
	
	\begin{prop}(\cite[Thm.1]{thevenaz})\label{P:thevenaz}
		Let $\mathcal{U},\mathcal{T}$ be two $G$-posets and $\phi:\mathcal{U} \rightarrow \mathcal{T}$ a map of $G$-posets such that, for all $\sigma \in \mathcal{T}$, $\phi^{-1}(\mathcal{T}_{\leq \sigma})$ is $G_\sigma$-contractible(that is, its geometric realization). Then $\phi$ is a $G$-homotopy equivalence(that is, on their realizations).
		The analogous statement holds true with $\mathcal{T}_{\leq \sigma}$ replaced by $\mathcal{T}_{\geq \sigma}$. 
	\end{prop}
It is worth noting that the strategy of proof of Proposition~\ref{P:thevenaz}, uses a weaker form of Theorem \ref{T:equi_whitehead}. 	
	
	We would like to also rewrite \cite[Cor.1.8]{quillen} in $G$-equivariant terms. In the following, given a $G$-poset $\mathcal{T}$, by a $G$-closed subset $S$ of $\mathcal{T}$, we mean a $G$-subposet that is closed under \textit{specialization}, that is $x' \leq x \in S \implies x' \in S$.
	
	\begin{prop}(cf. \cite[Prop.1.7]{quillen})\label{P:quillen_1}
		Let $\mathcal{U},\mathcal{T}$ be two $G$-posets and $\mathcal{Z} \subset \mathcal{U} \times \mathcal{T}$ a $G$-closed subset. Let $p_1:\mathcal{Z} \rightarrow \mathcal{U}$ denote the first projection. If the fiber $\mathcal{Z}_u:=\{t \in \mathcal{T}|(u,t) \in \mathcal{Z}\}$, is $G_u$-contractible for all $u \in \mathcal{U}$, then $p_1$ is a $G$-homotopy equivalence. 
	\end{prop}
	\begin{proof}
		We apply Proposition \ref{P:thevenaz} to the map $p_1$, which is a map of $G$-posets. Consider the two maps of posets \[f:\mathcal{Z}_u \rightarrow p_1^{-1}(\mathcal{U}_{\geq u}),\ g:p_1^{-1}(\mathcal{U}_{\geq u})\rightarrow \mathcal{Z}_u,\] given by $f(t):=(u,t)$ and $g(u',t):=t$. Note that $g$ is well defined since $(u',t) \in p_1^{-1}(\mathcal{U}_{\geq u})$ implies that $(u,t) \leq (u',t)$, and since $\mathcal{Z}$ is closed, $(u,t) \in \mathcal{Z}_u$. $f,g$ are $G_u$-equivariant, and one has $g\circ f(t)=t$ and $f\circ g(u',t) \leq (u',t)$. Hence by Lemma \ref{L:quillen1.3}, $f$ and $g$ are $G_u$-homotopy equivalences and hence by hypothesis on $\mathcal{Z}_u$, $p_1^{-1}(\mathcal{U}_{\geq u})$ is $G_u$-contractible. Applying Proposition \ref{P:thevenaz} finishes the proof.
	\end{proof}

	\begin{prop}(\cite[Lemma 1.9]{grayson})\label{T:lem1.9}
		Let $X$ be a $G$-simplicial complex, $\mathcal{T}$ a $G$-poset and for each $\sigma \in \mathcal{T}$, $X_\sigma \subset X$ be a subcomplex. Suppose the following conditions holds true:-
		\begin{itemize}
			\item[(i)] $\sigma \leq \tau$ $\implies \ X_\sigma \supset X_\tau$;
			\item[(ii)] every simplex of $X$ is a simplex of $X_\sigma$ for some $\sigma \in \mathcal{T}$ i.e. $X=\underset{\sigma \in \mathcal{T} }{\bigcup} X_\sigma$;
			\item[(iii)] each $X_\sigma$ is contractible;
			\item[(iv)] for each simplex $\gamma$ in $X$, the poset $\mathcal{T}_\gamma:=\{\sigma \in \mathcal{T}|\ \gamma \text{ is a simplex of } X_\sigma\}$ is $G_\gamma$-contractible. 
		\end{itemize}
		Then there is a natural $G$-equivariant homotopy equivalence between $|X|$ and $|\mathcal{T}|$. Moreover, if (iii) is strengthened to each $X_\sigma$ being $G_\sigma$-contractible, then $|X|$ and $|\mathcal{T}|$ are naturally $G$-homotopy equivalent.  
	\end{prop}
	
	\begin{proof}
		Since a $G$-simplicial complex and its barycentric subdivision are naturally $G$-homoto\-py equivalent to each other, without loss of generality, we can replace $X$ and the $X_\sigma$ by their barycentric subdivision. Hence we may assume $X$ is a poset and each of the $X_\sigma$ is a closed subset of $X$, and condition (iv) can be reformulated as the poset $\mathcal{T}_x:=\{\sigma \in \mathcal{T}| x \in \mathcal{T}_\sigma\}$ being $G_x$-contractible, for all $x \in X$. Consequently, we can form the closed $G$-subposet $Z \subset X \times \mathcal{T}$, defined by \[Z:=\{(x,\sigma)|x \in X_\sigma\}.\] The projection $p_1:Z \rightarrow X$ is a $G$-equivariant map and its fiber at $x \in X$ is precisely $\mathcal{T}_x$, which by hypothesis is $G_x$-contractible. Hence by Proposition \ref{P:quillen_1}, $p_1$ is a $G$-homotopy equivalence. Also since the fiber of the second projection $p_2:Z \rightarrow \mathcal{T}$ at $\sigma \in \mathcal{T}$ is $X_\sigma$, which is contractible, from the non-equivariant version of Proposition \ref{P:quillen_1}(which is \cite[Prop.1.7]{quillen}), we know that $p_2$ is a $G$-equivariant homotopy equivalence. If we denote by $p'_1$ the $G$-homotopy inverse of $p_1$ obtained by the same proposition, then $X \xrightarrow{p_1'} Z \xrightarrow{p_2} \mathcal{T}$ induces a $G$-equivariant homotopy equivalence between $X$ and $\mathcal{T}$.  
		
		Moreover, when each $X_\sigma$ is also $G_\sigma$-contractible, then $p_2$ is also a $G$-homotopy equivalence, and hence in this case $|X|$ and $|\mathcal{T}|$ become $G$-homotopy equivalent. 
	\end{proof}
	
	\begin{rem}
		Note that, in all of the ``$G$-equivariant" results stated in this section, one obtains a ``non-equivariant" version by choosing $G=1$. Later on, this is what we will mean by a non-equivariant version of any of the above results in this subsection. 
	\end{rem}

	\section{Bruhat-Tits building}
		
	We briefly recall the notation from the introduction. Let $K$ denote the function field of a geometrically irreducible curve $C$ over a finite field $\FF_q$ and let $\infty$ be a fixed closed point on $C$. Define $A \subset K$ as the subset of those functions that are regular outside $\infty$, and let $R$ be the discrete valuation ring $\mathcal{O}_{C,\infty}$ with $\varpi$ a choice of uniformizer of $R$. In this section, we recap the basic facts about the rank $(r-1)$ Bruhat-Tits building for $\GL_r(R)$. 
	We first describe it as a simplicial complex of homothety classes of $R=\mathcal{O}_{C,\infty}$-lattices inside a fixed rank $r$ $K$-vector space $W$. In a later subsection we provide an alternate description of this building as classes of rank $r$ vector bundles on the curve $C$ whose restriction to $\Spec(A)$ is a fixed projective $A$-module $P$, following \cite[\S4]{grayson}. We will see there how these two descriptions are equivalent to the classical way of describing the Bruhat-Tits building for $\GL_{r}(K_\infty)$. Also part of this section, is a discussion on the stabilizers of simplices with respect to $\Aut_A(P)$ and its principal congruence subgroups.    
	
	\begin{defi}
		Define the poset \[\mathbb{B}(W):=\{L \subset W \text{ is an }R\text{-submodule of full rank }r\},\]with inclusion as the partial order. $\ZZ$ acts cofinally on $\mathbb{B}(W)$ via \[n+L:=\varpi^{-n}L;\ n \in \ZZ,L \in \mathbb{B}(W).\] Then \[\mathcal{B}_\bullet(W):=\langle \mathbb{B}(W)\rangle.\](See \eqref{E:poset-simplicial}(2) for the definition of $\langle \mathbb{E}\rangle$ for a $\mathbb{Z}$-poset $\mathbb{E}$) As before, we denote by $\mathcal{B}_i(W)$ the set of $i$ simplices. The complex $\mathcal{B}_\bullet(W)$ is called the \textit{Bruhat-Tits building of rank $r$}. 
	\end{defi}
	
	\begin{rem}
		 The vertex set $\mathcal{B}_0(W)$ consists of the homothety classes of $R-$lattices $L \subset W$ of rank $r$, and are denoted by $\langle L \rangle$. Similarly, $\{ \langle L_0\rangle,\ldots,\langle L_{d}\rangle\} \in \mathcal{B}_d(W)$ is a $d$-simplex if the union of the classes forms a chain. Equivalently, the $L_i$'s can be ordered and scaled so that $L_0\supsetneq L_1\supsetneq\ldots\supsetneq L_d\supsetneq \varpi L_0$.
	\end{rem}
	
	\subsection{Congruence subgroups}
	
	Here we follow the treatment from \cite[\S 3.4]{gebhard1}. Note that the notion of congruence subgroup is developed there in the rank 2 case but many results easily generalize to higher rank. Fix a $K$-basis $\{e_i\}^r_{i=1}$ of $W$, and let $\GL_r(A)$ be the group of $A$-automorphisms of $
	\bigoplus_{1 \leq i \leq r} Ae_i$.

	\begin{defi}
		A subgroup $\Gamma \leq GL(W)$ is called a \textbf{congruence subgroup} if there is a non zero ideal $I \subset A$ such that $\Gamma \supset \Gamma(I):=\Ker(\GL_r(A) \rightarrow \GL_r(A/I))$ is of finite index.
	\end{defi}
	
	Recall that two subgroups $G',G'' \leq G$ are said to be  \textit{commensurable} if $G' \cap G''$ is a finite index subgroup inside both $G'$ and $G''$. 
	
	\begin{lem}\label{L:commensurable}
		Let $\Gamma',\Gamma'' \subset \GL(W)$ be two congruence subgroups. Then $\Gamma'$ and $\Gamma''$ are commensurable.     
	\end{lem}
	
	\begin{proof}
		Suppose $\Gamma'\supset \Gamma(I)$ and $\Gamma'' \supset \Gamma(J)$ are of finite index. Then both $\Gamma',\Gamma''$ contain $\Gamma(IJ)$ as a common finite index subgroup. Consequently, also $\Gamma' \cap \Gamma''\supset\Gamma(IJ)$ is a finite index subgroup of both $\Gamma'$ and $\Gamma''$. 
	\end{proof}
	
	We also give below an equivalent definition of a congruence subgroup. This is the definition used in \cite[V.2]{Gekeler} for an arithmetic group. 
	
	\begin{lem}\label{L:ari_grp}
		Let $\Gamma \leq \GL(W)$ be a subgroup. Then $\Gamma$ is a congruence subgroup if and only if there exist a projective $A$-module $\Lambda \subset W$ of rank $r$ and an ideal $I \subset A$ such that $\Aut_A(\Lambda,I):=\Ker(\Aut_A(\Lambda) \rightarrow \Aut_A(\Lambda/I\Lambda)) \subset \Gamma \subset \Aut_A(\Lambda)$.
	\end{lem}
	
	\begin{proof}
		The proof of \cite[Prop. 3.14]{gebhard1} verbatim holds true for arbitrary ranks as well. 
	\end{proof}
	
	To ease notation, we will abbreviate $\Gamma_\Lambda:=\Aut_A(\Lambda)$ and $\Gamma_\Lambda(I):=\Aut_A(\Lambda,I)$.

	Note that $\GL(W)$ acts on $\mathcal{B}_\bullet(W)$ naturally, regarding an $R$-lattice as an $R$-submodule of maximal rank inside $W$. In particular, there is an induced action of $\Gamma_P$ on $\mathcal{B}_\bullet(W)$, where $P$ is a fixed projective submodule of $W$. Also we have the analogue of \cite[Prop. 3.16]{gebhard1}, which says that a congruence subgroup stabilizes a simplex if and only if it stabilizes all the vertices of the simplex.
	
	First we state a result on calculating distances of vertices in $|\mathcal{B}_\bullet(W)|$. Recall that for two vertices $\langle L\rangle,\langle U\rangle$ in $\mathcal{B}_\bullet(W)$, $d(\langle L\rangle,\langle U\rangle)$ is the length of a minimal path $x_0:=\langle L\rangle,x_1:=\langle L_1\rangle,\ldots,x_n:=\langle U\rangle$ between $\langle L\rangle$ and $\langle U\rangle$, such that $\langle L_{i}\rangle$ and $\langle L_{i+1}\rangle$ form an edge in $\mathcal{B}_\bullet(W)$.  
	
	\begin{lem}(\cite[Ch. 3, Prop. 3.4]{de})\label{L:dist}
		Let $L,U$ be two $R$-lattices such that $L \supset U \supset \varpi^r L$ and $\varpi L \not\supset U \not\supset \varpi^{r-1}L$. Then $d(\langle L\rangle,\langle U\rangle)=r$.	\end{lem}
		
	\begin{prop}\label{L:ari_dist}
		Let $\Gamma \leq \GL(W)$ be a congruence subgroup, $\gamma \in \Gamma\setminus \{1\}$ and $v \in \mathcal{B}_\bullet(W)$ a vertex. Then $\det(\Gamma) \subset \mathbb{F}^{*}_{q}$ is a finite group and $d(v,\gamma v) \neq 1$.  
	\end{prop} 
	
	\begin{proof}
		Since $\Gamma \supset \Gamma(N)$ for some $N$, as a finite index subgroup, and $\det(\Gamma(N))=\{1\}$, we have $\det(\Gamma) \subset K^\ast$ is a finite subgroup. But the only finite subgroups of $K^\ast$ are those of $\mathbb{F}^\ast_q \subset K^\ast$, hence the first statement. 
		
		As in \cite[Prop. 3.16]{gebhard1}, we follow the argument leading to corollary in \cite[Pg. 75]{serre}, to prove the second assertion. Suppose $v$ is represented by a lattice $L$. Let $\{e_i\}^r_{i=1}$ be an $R$-basis spanning $L$ such that $\{\pi^{a_i}e_i\}^r_{i=1}$ spans the lattice $L\gamma$, with $a_i \leq a_{i+1}$. So $\gamma=\gamma_1\delta\gamma_2$, where $\delta$ is the diagonal matrix with entries $\pi^{a_i}$ and $\gamma_1,\gamma_2$ are matrices in $GL_r(R)$. Since the determinant of $\gamma$ is a unit, from taking valuations we find that $\sum a_i = 0$. Because the $a_i$ are increasing in $i$, either all $a_i=0$ and $L=L\gamma$, or $a_1\le-1$ and $a_r\ge1$, so that $a_r-a_1\ge2$. Now by Lemma \ref{L:dist}, $d(L,\gamma L)=a_r - a_1$, and the assertion follows.
	\end{proof}
	
	Consequently we have
	
	\begin{cor}\label{C:ari_grp}
		Let $\Gamma$ be as above and $s=(x_0,\ldots,x_d)$ be a simplex in $\cB_\bullet(W)$. If $\gamma \in \Gamma$ stabilizes $s$, then it stabilizes all the vertices of $s$. In particular, $\Gamma_s=\underset{0 \leq i \leq d}{\bigcap} \Gamma_{x_i}$. 
	\end{cor}
	\begin{proof}
		By \ref{L:ari_dist}, $\gamma(x_i) \in s$ and is not equal to any of $x_j$ for $j \neq i$. Hence $\gamma(x_i)=x_i$ for all~$i$. 
	\end{proof}
	
	\subsection{$\mathcal{B}_\bullet(W)$ as a building of vector bundles}
	
	Let $\widetilde{W}$ denote the constant sheaf on $C$ associated to $W$. Fix the ample line bundle $\mathcal{O}(1):=\mathcal{O}(\infty)$, the dual of the vanishing ideal at $\infty$, which is a line bundle. 
	
	Recall that $P$ is a fixed projective $A$-module of rank $r$ inside $W$. We will be concerned with vector bundles whose sections over $\Spec(A)$ are $P$ and whose stalk at $\infty$ varies over rank $r$ lattices, as explained below.
	
	\begin{lem}(cf. \cite[\S2.8]{wiedmann})\label{L:uni_vb}
		Let $L$ be a rank $r$ $R$-lattice inside $W$. Denote by $\widetilde P$ the coherent sheaf over $\Spec (A)$ inside $\widetilde W$ with global sections $P$. Then there exists a unique rank $r$ vector bundle $E$ over $C$ inside $\widetilde W$ such that 
		\begin{align}\label{eq:def-E-P-L} E_{|\Spec(A)}= \widetilde{P}\hbox{ and } E_{\infty}=L
		\end{align}
	\end{lem}
	
	\begin{proof}
		Let $i:\Spec(A)=C-\{\infty\} \rightarrow C$ and $j:\Spec(\mathcal{O}_{C,\infty}) = \{\eta,\infty\} \rightarrow C$ be the cannonical morphisms, where $\eta$ is the generic point of $C$. Consider the sheaf $E:=i_*(\widetilde{P}) \cap j_*(\widetilde{L}) \subset \widetilde{W}$. 
		$E$ is a coherent sheaf which is locally free, for all its stalks are free modules over the corresponding local rings. Thus we have a vector bundle satisfying the hypothesis of the lemma. 
		Now let $E,E' \subset \widetilde{W}$ be two vector bundles both satisfying \eqref{eq:def-E-P-L}. 
		Then $E/E \cap E' \simeq (E + E')/E'$, with the right hand side having all its stalks zero. Hence $E \subset E'$. By symmetry of the argument, we get $E=E'$. Hence the uniqueness of such a vector bundle. 
	\end{proof}	

	\begin{defi}
		Let $L \subset W$ be a $R$-module of rank $r$ inside $W$. Denote by $\mathcal{E}(P,L)$ the unique rank $r$ vector bundle on $C$ satisfying the hypothesis Lemma \ref{L:uni_vb}.
	\end{defi}
	\begin{rem}\label{R:pair_asso}
		\begin{itemize}
			\item[(i)] It is easy to deduce from the proof above that $H^0(C,\mathcal{E}(P,L))=P \cap L$, where the intersection is inside $W$.
			\item[(ii)] It is instructive to point out that the association $(P,L) \mapsto \mathcal{E}(P,L)$, from the category of pairs $(P,L)$ as above to the category of rank $r$ vector bundles on $C$, commutes with taking duals and forming tensor products. That is, $\mathcal{E}(P,L)^\vee = \mathcal{E}(P^\vee,L^\vee)$ and $\mathcal{E}(P_1,L_1) \otimes_{\mathcal{O}_C} \mathcal{E}(P_2,L_2)=\mathcal{E}(P_1 \otimes_A P_2,L_1\otimes_R L_2)$.
		\end{itemize}
		
	\end{rem}
	
	\begin{defi}
		
		Define the poset \[\mathbb{E}(P):=\{\mathcal{E}(P,L)|\ L \subset W \text{ is an } R\text{-submodule of rank }r\}.\] $\ZZ$-acts cofinally on it via \[n+\mathcal{E}(P,L):=\mathcal{E}(P,L)(n):=\mathcal{E}(P,L) \otimes \mathcal{O}(1)^{\otimes n}=\mathcal{E}(P,\varpi^{-n}L) ;\ n \in \ZZ,\ \mathcal{E}(P,L) \in \mathbb{E}(P).\] 
		Define the simplicial complex \[\Chi_{\bullet}(P):=\langle \EE(P) \rangle \] and $\Chi_d(P)$ denotes the set of $d$-simplices of $\Chi_\bullet(P)$, for $d \geq 0$. The top dimensional simplices, in $\Chi_{r-1}(P)$, are called \textit{chambers}. 
	\end{defi}
	
	\begin{rem}
		
		Note that the set of vertices $\Chi_0(P)$ is just the $\ZZ$-orbit of vector bundles $\mathcal{E}(P,L)$ and $\{\langle \mathcal{E}_0 \rangle,\ldots,\langle \mathcal{E}_{d} \rangle\} \in \Chi_d(P)$ is a $d$-simplex, if the vertices can ordered and scaled so that $\mathcal{E}_0 \subsetneq \ldots \subsetneq \mathcal{E}_{d} \subsetneq \mathcal{E}_0(1)$. 
		
	\end{rem}

	There is a natural left action of $\Gamma_P$ on $\Chi_\bullet(P)$, defined by $g.\mathcal{E}(P,L):=\mathcal{E}(P,gL)$, for $g \in \Gamma_P$.

	A priori it is not clear whether the building $\Chi_\bullet(P)$ is independent of $P$ as a simplicial complex. The following lemma shows that $\Chi_\bullet(P)$ is in fact isomorphic to the Bruhat-Tits building $\mathcal{B}_\bullet(W)$, as simplicial complexes. 
	
	\begin{lem}\label{L:isom_BT}
		The map $\mathbb{B}(W) \rightarrow \EE(P)$ sending $L \mapsto \mathcal{E}(P,L)$ is bijective and $\ZZ$-equivariant. In particular, it induces an isomorphism of simplicial complexes $\mathcal{B}_{\bullet}(W) \simeq \Chi_{\bullet}(P)$ that is $\Gamma_P$-equivariant.   
	\end{lem}
	
	\begin{proof}
		It is enough to observe that $\mathcal{E}(P,L)(-1)=\mathcal{E}(P,\varpi L)$. From this and Lemma \ref{L:uni_vb}, we obtain the isomorphism of the simplicial complexes. $\Gamma_{P}$-equivariance is also clear. 
	\end{proof}
	
	Consequently, most notions involving the simplicial structure on $\Chi_\bullet(P)$ have an equivalent notion in $\mathcal{B}_{\bullet}(W)$ and vice versa. 
	
	\begin{rem}
		In many references(eg. \cite{gebhard1},\cite{drin},\cite{de}), the Bruhat-Tits building $\mathcal{BT}^{\,r}_\bullet$ of rank $r$ is defined as the simplicial complex whose
		
		(1) \textit{vertices} are given by homothety classes $\langle L\rangle$ of rank $r$ $\widehat{R}$-lattices $L \subset K^{r}_{\infty}$.
		
		(2) \textit{$d$-simplices} are given by $\{\langle L_0\rangle,\ldots,\langle L_d\rangle\}$ satisfying, upto ordering and scaling $L_0\supsetneq L_1\supsetneq\ldots\supsetneq L_d\supsetneq \varpi L_0$.
		
		The 1-1 correspondence \[\{R-\text{lattices} \subset W\} \xrightarrow{1-1} \{\widehat{R}-\text{lattices} \subset \widehat{W}\}\] \[L \mapsto L \otimes_R \widehat{R}, \] induces a natural isomorphism $\mathcal{B}_\bullet(W) \rightarrow \mathcal{BT}^{\,r}_\bullet$ of simplicial complexes and is also $GL_r(K)$ equivariant.   
	\end{rem}

	\subsection{Labeling the vertices of $\Chi_\bullet(P)$.}
	
	Later we will consider the homological chain complex of the Bruhat-Tits building, for which purpose it is important to consider the orientation on the simplices. With this in view, we define a labeling on the vertices of $\Chi_\bullet(P)$ generalizing the notions of \textit{even} and \textit{odd} vertices of the Bruhat-Tits tree as in \cite[Ch.2,Prop.5]{serre}. 
	
	For the following lemma, we fix a basis $\{e_1,\ldots,e_r\}$ of $W$ and consider the lattice $L_0:=Re_1 + \ldots +Re_r \subset W$.
	
	\begin{lem}
		Let $\mathcal{E}(P,L) \in \mathbb{E}(P)$ be a vector bundle. Then $\det(\mathcal{E}(P,L)) = \mathcal{L}_0(n)$ for $\mathcal{L}_0 = \mathcal{E}(\det(P),\det{L_0})$ and some $n \in \ZZ$. 
	\end{lem}
	
	\begin{proof}
		By Remark \ref{R:pair_asso}, we have that $\det(\mathcal{E}(P,L)) = \mathcal{E}(\det(P),\det(L))$. But by elementary divisors theorem, $\det(L)=\varpi^{n}\det(L_0)$ for a suitable integer $n \in \ZZ$ and hence the lemma.
	\end{proof}
	
	The above lemma together with the fact that for a rank $r$ vector bundle $\mathcal{E} \in \mathbb{E}(P)$, $\det(\mathcal{E}(n)) = \det(\mathcal{E})(rn)$, motivates us to define the following
	
	\begin{defi}
		\[Types \ = \{\mathcal{L}_i:=\mathcal{L}_0(-i)=\mathcal{E}(\det(P),\varpi^{i}R)\ |\ 0 \leq i \leq r-1\}.\]A vertex $\langle \mathcal{E} \rangle$ is said to be of \textit{type $i$}, if $i$ is the unique integer in $[0,r-1]$ such that $\det (\mathcal{E}(n)) = \mathcal{L}_i $ for some $n\in\ZZ$. In this case we will also denote by $type(\langle \mathcal{E}\rangle) := i$. 
	\end{defi}
	
	\begin{rem}\label{R:type_i_vertex}
		\begin{itemize}
			
			\item[(i)] From the proof of Proposition \ref{L:ari_dist}, if $\gamma \in \Gamma_P$, then there exists a $K$-basis $\{f_i\}^{r}_{i=1}$ s.t. $L=\underset{1 \leq j \leq r}{\sum}Rf_j$ and $\gamma L = \underset{1 \leq j \leq r}{\sum}\varpi^{a_j}Rf_j$ with $\underset{1 \leq j \leq r}{\sum} a_j = 0$. In particular, we see that $\det(L) = \det(\gamma(L))$ or $\gamma L$ has the same type as that of $L$. 
			\item[(ii)] By Lemma \ref{L:isom_BT}, we will also define the type of a vertex $x \in \mathcal{B}_\bullet(W)$ as the type of the corresponding vertex in $\Chi_\bullet(P)$. 
		\end{itemize}
		
	\end{rem}

	\subsection{Stabilizers of simplices}

	Let $\Gamma$ be a congruence subgroup. By the Lemma \ref{L:ari_grp}, there is a projective $A$-module $M$ of rank $r$ inside $W$ such that $\Gamma \subset \Gamma_M$ is a finite index subgroup. We can assume without loss of generality that $M=P$, where $P$ is a fixed projective $A$-module as in the last section. Here we would like to study the finiteness of the stabilizer subgroup $\Gamma_e$, for $e$ any simplex in the Bruhat-Tits building, and when possible would also like to obtain an explicit description of the stabilizer. Recall that the stabilizer subgroup $\Gamma_e$ of $\Gamma$, is the set of all $\gamma \in \Gamma$ such that $\gamma e=e$. By Lemma \ref{L:isom_BT}, without loss of generality, we can consider the avatar $\Chi_\bullet(P)$ of the Bruhat-Tits building. We have the following result 
	
	\begin{prop}(cf. \cite[\S 5]{grayson})\label{P:stab-finite}
		$\Gamma_e$ is a finite group. 
	\end{prop}
	
	\begin{proof}
		
		From the Corollary \ref{C:ari_grp}, any $\gamma \in \Gamma_e$ fixes all the vertices of $e$ as well, and hence it is enough to assume $e$ is a vertex. We assume that $e$ is the class of a vector bundle $\mathcal{E}$, with $\langle \mathcal{E}\rangle \in \Chi_\bullet(P)$.   
		
		Clearly the stabilizer $(\Gamma_P)_{e}=\Aut(\mathcal{E})$. But this is a finite group, since $\Aut(\mathcal{E}) \subset \End(\mathcal{E}) \simeq H^0(C,\mathcal{E}^{\vee} \otimes \mathcal{E})$ and the latter is a finite dimensional $\mathbb{F}_q$ vector space and hence is a finite set. Since $\Gamma \subset \Gamma_P$, $\Gamma_e$ is also a finite set.

	\end{proof}
	
	Moreover, when the congruence subgroup is principal, that is $\Gamma=\Gamma_P(I)$ for some proper non-zero ideal $I \subset A$, then we can explicitly describe the stabilizer group $\Gamma_e$, as we shall now explain. For any subgroup $G \subset \GL(W)$ and $L \subset W$ an $R$-lattice of maximal rank, we denote by $G_L$ the subgroup of $G$ fixing $L$. 
	
	\begin{prop}\label{P:nilp_stab}
		Let $L$ be an $R$-lattice of rank $r$ and $I \subset A$ be a proper non-zero ideal. 
		\begin{itemize}
			\item[(i)] Then \[H^{0}(C,\mathcal{E}(P,L)^\vee \otimes_{\mathcal{O}_C} \mathcal{E}(P,L)) = \End_A(P)_L.\]
			\item[(ii)] Let $\mathcal{I}:=\mathcal{E}(I,R)$ and for any vector bundle $\mathcal{F}$ on $C$ and $n \in \ZZ$, denote by $\mathcal{F}(I^n):=\mathcal{F} \otimes_{\mathcal{O}_C} \mathcal{I}^{\otimes n}$. For an element 
			\[\gamma \in H^{0}(C,\mathcal{E}(P,L)^\vee \otimes_{\mathcal{O}_C} \mathcal{E}(P,L)(I)) \subset H^0(C,\mathcal{E}(P,L)^\vee \otimes_{\mathcal{O}_C} \mathcal{E}(P,L))=\End_A(P)_L,\] we have that $\gamma P \subset IP$ and that $\gamma$ is nilpotent. 
		\end{itemize}
	\end{prop}
	\begin{proof}
		(i) From the Remark \ref{R:pair_asso}(i,ii), it is clear that \[H^{0}(C,\mathcal{E}(P,L)^\vee \otimes_{\mathcal{O}_C} \mathcal{E}(P,L)) = \End_A(P) \cap \End_R(L),\] where the intersection on the right side is inside $\End_K(W)$. Clearly $\End_A(P) \cap \End_R(L)=\End_A(P)_L$.
		
		(ii) Again by Remark \ref{R:pair_asso}(ii), \[\mathcal{E}(P,L)^\vee \otimes_{\mathcal{O}_C} \mathcal{E}(P,L)(I) = \mathcal{E}((P^\vee) \otimes_A IP,L^\vee \otimes_R L) = \mathcal{E}(\Hom_A(P,IP),\End_R(L)).\] Hence by part (i) of the same remark, $H^{0}(C,\mathcal{E}(P,L)^\vee \otimes_{\mathcal{O}_C} \mathcal{E}(P,L)(I)) = \Hom_A(P,IP) \cap \End_R(L)$, and hence we can conclude that $\gamma P \subset IP$. 
		It follows that for any $n \geq 0$, $\gamma^n \in \Hom_A(P,I^nP) \cap \End_R(L)=H^{0}(C,\mathcal{E}(P,L)^\vee \otimes_{\mathcal{O}_C} \mathcal{E}(P,L)(I^n))$. By \cite[Ch.7,Prop.4.4]{liu}, the line bundle $\mathcal{I}^\vee$ is ample(since $\deg(\mathcal{I}) < 0$). Hence by Serre vanishing theorem(\cite[Ch.5,Prop.3.6]{liu}), $H^{0}(C,\mathcal{E}(P,L)^\vee \otimes_{\mathcal{O}_C} \mathcal{E}(P,L)(I^n))=0$ for $n \gg 0$. Hence $\gamma$ is nilpotent.   
	\end{proof}
	
	\begin{cor}\label{C:prin_stab}
		$\Gamma_P(I)_L = 1 + H^{0}(C,\mathcal{E}(P,L)^\vee \otimes_{\mathcal{O}_C} \mathcal{E}(P,L)(I))$, where the sum on the right is inside $\End_K(W)$. Moreover $\Gamma_P(I)_L $ is a finite $p$-group.
	\end{cor}
	\begin{proof}
		By Proposition \ref{P:nilp_stab}, the inclusion $\supset$ is clear. For the other direction, it is enough to observe that for any $\gamma \in \Gamma_P(I)$, $\gamma-1 \in \Hom_A(P,IP)_L$. For the last assertion note that $\Gamma_P(I)_L $ is finite by Proposition~\ref{P:stab-finite} and since any $n\in H^{0}(C,\mathcal{E}(P,L)^\vee \otimes_{\mathcal{O}_C} \mathcal{E}(P,L)(I))$ is nilpotent, we have $(1+n)^{p^m}=1$ as soon as $n^{p^m}=0$.
	\end{proof}
	
	\section{An equivariant homotopy between the unstable subcomplex and the (rational) Tits building(after Grayson-Quillen)}

	In this section we will show the promised homotopy equivalence between certain subcomplex(es) of the Bruhat-Tits building with the Tits building $|\mathcal{T}_\bullet(W)|$. It seems prudent at the outset to take into account the action of congruence subgroups in the homotopies to be discussed, for which the results from \S2.4 become relevant. 
	
	We will be using the description of Bruhat-Tits building with vertices as equivalence classes of $R$-lattices in $W$(recall that $R=\mathcal{O}_{C,\infty}$). We denoted this version of the building by $\cB_\bullet(W)$ in \S3. Also, we had a $\ZZ$-poset $\BB(W)$ which consisted of all $R$-lattices of maximal rank inside $W$, with $n\in \ZZ$ acting on $L \in 
	\BB(W)$ via $n+L:=\varpi^{-n}L$. Furthermore we denoted by $\mathcal{E}(P,L)$ the unique vector bundle on $C$, such that its sections over $\Spec(A)$ are $P$, for a fixed maximal rank projective $A$-module $P$ inside $W$, and its stalk at $\infty$ is an $R$-lattice $L \subset W$. The congruence subgroups we will be concerned with in this section are the principal ones, namely we let $I \subset A$ be a proper non-zero ideal and set $\Gamma=\Gamma_P(I)$, so that $\Gamma \trianglelefteq \Gamma_P=\Aut_A(P)$. 
	
	\subsection*{Tits building and the Steinberg module}
	
	We recall here for the benefit of the reader the definition of the Tits building for $W$ and the Steinberg module associated to it. 
	
	\begin{defi}
		The \textit{Tits building for $W$}, denoted by $\mathcal{T}_\bullet(W)$, is the simplicial complex based on the poset of \textit{proper, non-zero} $K$-subspaces of $W$. 
	\end{defi}
	
	In particular, the set of vertices $\mathcal{T}_0(W)$ consists of proper non-zero subspaces of $W$ and $d$-simplices $\mathcal{T}_d(W)$ consists of proper chains of length $d+1$ of such subspaces. Also the group $\GL(W)$ acts naturally on $\mathcal{T}_\bullet(W)$ via simplicial automorphisms. We will denote a vertex in $\mathcal{T}_\bullet(W)$ by the subspace representing it. 
	
	 $\mathcal{T}_\bullet(W)$ is often also called as a spherical building due to the fact that an apartment of such a building is the Coxeter complex associated with the symmetric group $\mathfrak{S}_r$, hence can be regarded as triangulations of an  $(r-2)$ sphere. This is clarified by the following important result known as the \textit{Solomon-Tits theorem}
	
	\begin{thm}(\cite[\S2]{solomon_quillen})\label{solomon_quillen}
		If $r\geq 2$, then $\mathcal{T}_\bullet(W)$ has the homotopy type of a bouquet of $(r-2)$-spheres, that is, it is a wedge sum of a family of $\mathbb{S}^{r-2}$. 
	\end{thm}
	
	Consequently, the (reduced) homology is non vanishing only in the top dimension. Also the action of $\GL(W)$ on $\mathcal{T}_\bullet(W)$ induces an action of $\GL(W)$ on the homology.
	
	\begin{defi}\label{D:steinberg}
		The $\ZZ[\GL(W)]$-module $\St(W):=\tilde{H}_{r-2}(\mathcal{T}_\bullet(W),\ZZ)$ is called the \textit{Steinberg module for W}
	\end{defi}
	
	\begin{rem}
		Due to the work of Ash \'{e}t al in \cite{Ash_1,Ash_2}, one has an explicit description of $\St(W)$ via a set of generators (known as \textit{modular symbols}) and relations, with a natural action of $\GL(W)$ on the set of generators.  
	\end{rem}

	\subsection{Contractible subcomplexes of the $\Gamma$-unstable complex} 
	
	Let $0 \subsetneq W_1 \subsetneq W$ be a proper subspace, so that $\sigma=(W_1) \in \mathcal{T}_\bullet(W)$. We define the following subposet of $\BB(W)$, \[\BB(W)_\sigma:= \{L \in \mathbb{B}(W)| \exists \ g \in \Gamma_L \setminus \{1\} \text{ s.t. }g_{|W_1}=1_{|W_1}.\}\] Clearly the $\ZZ$-action on $\BB(W)$ induces a cofinal $\ZZ$-action on $\BB(W)_{\sigma}$. We form the simplicial complex \[\cB_\bullet(W)_\sigma, \] whose vertices are the $\ZZ$-orbits of elements of $\BB(W)_\sigma$, and $\gamma=\{L_0, \ldots ,L_d\}$ gives rise to a $d$-simplex precisely when $\Gamma_\gamma=\underset{0 \leq i \leq d}{\bigcap} \Gamma_{L_i}\neq \{1\}$ and there is a $g \in \Gamma_\gamma \setminus \{1\}$ such that $g_{|W_1}=1_{W_1}$.
	
	\begin{rem}
		Note that the complex $\cB_\bullet(W)_\sigma$ is not the same as $\langle \BB(W)_\sigma \rangle$. The latter is a possibly larger complex. 
	\end{rem}
	
	The main results of this section are motivated from \cite[Thm.4.1 and Cor.4.2]{grayson}. As alluded to before, also the proof is structured around the corresponding results in op.cit.
	
	In the following theorem, we denote by $\sigma=(W_1) \in \mathcal{T}_\bullet(W)$ the vertex corresponding to the proper subspace $0 \subsetneq W_1 \subsetneq W$. 
	
	\begin{thm}(cf. \cite[Thm.4.1]{grayson})\label{T:un_contr}
		$|\cB_\bullet(W)_{\sigma}|$ is contractible. 
	\end{thm}
	
	\begin{proof}
		Denote by $P_1:=P \cap W_1$ and let \[\alpha:\BB(W)_{\sigma} \rightarrow \BB(W/W_1)\] be the map  sending \[L \mapsto L/(L \cap W_1).\]We denote by $[\alpha]:\cB_\bullet(W)_\sigma \rightarrow \cB_\bullet(W/W_1)$ the induced simplicial map. 
		
		To ease the notation below, for any lattice $L' \subset W/W_1$(resp. $L_1 \subset W_1$), we denote by $\mathcal{E}_{L'}$(resp. $\mathcal{E}_{L_1}$) the vector bundle $\mathcal{E}(P/(P\cap W_1),L')$(resp. $\mathcal{E}(P_1,L_1)$). Since $P_1 \subset P$ is a saturated $A$-submodule, the quotient $P/P_1$ is also a projective $A$-module. We fix a splitting $P \simeq P_1 \oplus P/P_1$, consequently also fixing a splitting $W \simeq W_1 \oplus W/W_1$. Also fix an $R$-lattice $\widetilde{L} \subset W_1$ of maximal rank, so that we can form the vector bundle $\widetilde{\mathcal{E}}:=\mathcal{E}_{\widetilde{L}}$.
		
		Now define an augmentation map(see \ref{D:augmentation}) \[\epsilon:\BB(W/W_1) \rightarrow \ZZ\] by \[L' \mapsto \min\{m| H^0(C,\mathcal{E}^\vee_{L'} \otimes \widetilde{\mathcal{E}}(I + m\infty)) \neq \{0\}\} + 1. \] The convention regarding the twist $(I+m\infty)$ is as in Proposition \ref{P:nilp_stab}(ii). This allows us to define \[\beta:\BB(W/W_1) \rightarrow \BB(W)\] given by \[L' \mapsto \varpi^{-\epsilon(L')}\tilde{L} \oplus L',\]where the direct sum above is with respect to the splitting of $P$ fixed above.

		\underline{\textit{Claim 1.} } $\beta(L') \in \BB(W)_{\sigma}$. Moreover, given a simplex $\gamma=\{\langle L'_0\rangle,\ldots,\langle L'_d\rangle\}$ in $\cB_\bullet(W/W_1)$, $\beta(\gamma)$ is a simplex in $\cB_\bullet(W)_\sigma$. 
		
		\begin{proof}
			From Corollary \ref{C:prin_stab}, we know that for any $R$-lattice $L \subset W$ of maximal $R$-rank 
			\[\Gamma_L=1+H^0(C,\mathcal{E}^\vee(P,L) \otimes \mathcal{E}(P,L)(I)).\] If we set $\mathcal{E}'_L:=\mathcal{E}_{L/(L \cap W_1)}$ and $\mathcal{E}_{L,1}:=\mathcal{E}_{L \cap W_1}$, then we have a subsheaf \[(\mathcal{E}'_{L})^\vee \otimes \mathcal{E}_{L,1}(I) \hookrightarrow \mathcal{E}^\vee(P,L) \otimes \mathcal{E}(P,L)(I).\] First we show that $\Gamma_{\beta(L')} \neq 1$, for which it is sufficient to show that $H^{0}(C,(\mathcal{E}'_{\beta(L')})^\vee \otimes \mathcal{E}_{\beta(L'),1}(I)) \neq 0$. Now $\mathcal{E}'_{\beta(L')}=\mathcal{E}_{L'}$ and $\mathcal{E}_{\beta(L'),1}=\widetilde{\mathcal{E}}(\epsilon(L')\infty)$. Hence by definition of $\epsilon$, we  have the non-vanishing of $H^{0}(C,\mathcal{E}^\vee_{L'} \otimes \widetilde{\mathcal{E}}(I + \epsilon(L')\infty))$, and consequently also the non-vanishing of $H^0(C,\mathcal{E}^\vee(P,L) \otimes \mathcal{E}(P,L)(I))$. Also any element $g' \in H^{0}(C,\mathcal{E}^\vee_{L'} \otimes \widetilde{\mathcal{E}}(I + \epsilon(L')\infty))$, acts as the zero matrix on $W_1$, so that $g=1+g'$ acts as the identity on $W_1$. This proves the first part of the claim 1.
			
			For the second part, consider the $d$-simplex in $\cB_\bullet(W/W_1)$ associated with the chain $(L'_0 \subsetneq \ldots \subsetneq L'_d \subsetneq \varpi^{-1}L'_0)$. To show that $(\beta(L'_0) \subsetneq \ldots \subsetneq \beta(L'_d) \subsetneq \varpi^{-1}\beta(L'_0))$ gives rise to a simplex in $\cB_\bullet(W)_\sigma$, it suffices to furnish a non-zero element in the intersection of the space of sections 
			\begin{equation}\label{E:beta_simp}
				\underset{0 \leq i\leq d}{\bigcap} \h^0(C,\mathcal{E}^\vee_{L'_i} \otimes \widetilde{\mathcal{E}}(I + \epsilon(L'_i)\infty)),
			\end{equation}
			as in the proof of the first part. The intersection above takes place inside $\End(W)$.
			
			Note that we have the following relations \[\mathcal{E}^\vee_{L'_d} \subset \ldots \subset \mathcal{E}^\vee_{L'_0} \text{ and } \ \epsilon(L'_d)-1 \leq \epsilon(L'_0) \leq \ldots \leq \epsilon(L'_d).\]The second set of inequalities result from the definition of the augmentation map. Consequently the intersection in \eqref{E:beta_simp} contains $ \h^0(C,\mathcal{E}^\vee_{L'_d} \otimes \widetilde{\mathcal{E}}(I + (\epsilon(L'_d) -1)\infty))$. From the definition of $\epsilon$ clearly the latter space is a non-zero $\FF_q$-vector space. Hence we can conclude as before that $(\beta(L'_0) \subsetneq \ldots \subsetneq \beta(L'_d) \subsetneq \varpi^{-1}\beta(L'_0))$ determines a simplex in $\cB(W)_\sigma$. 
		\end{proof}  
		
		We denote the induced simplicial map from $\beta$ by $[\beta]:\cB_\bullet(W/W_1) \rightarrow \cB(W)_{\sigma}$.  
		
		Now define a new map 
		\[\hat{\epsilon}: \BB(W)_\sigma \rightarrow \ZZ\]as the composition $\BB(W)_\sigma \xrightarrow{\alpha} \BB(W/W_1) \xrightarrow{\epsilon} \ZZ$. Since $\alpha$ is order preserving and $\ZZ$-equivariant, $\hat{\epsilon}$ is an augmentation map.
		
		Continuing the line of thought from the proof of \cite[Thm.4.1]{grayson}, we introduce for any integer $n \in \ZZ$, a map of posets \[G_n:\BB(W)_{\sigma} \rightarrow \BB(W),\] by \[L \mapsto L + \varpi^{-n-\hat\epsilon(L)}\widetilde{L}.\] 
		
		\underline{\textit{Claim 2.}} For any simplex $\gamma \in \cB_\bullet(W)$ and any $n \in \ZZ$, there is an inclusion of subgroups \[\{g \in \Gamma_\gamma| g_{|W_1}=1_{W_1}\} \subset \{g \in \Gamma_{G_n(\gamma)}| g_{|W_1}=1_{W_1}\}.\] Consequently $G_n$ has target inside $\BB(W)_{\sigma}$ and induces a simplicial map, denoted by $[G_n]:\cB_\bullet(W)_{\sigma} \rightarrow \cB_\bullet(W)_{\sigma}$.
		
		\begin{proof}
			Let $\gamma=\{L_0 , \ldots, L_d\}$ and $g \in \Gamma_\gamma \setminus \{1\}$ be given such that $g_{|W_1}=1_{W_1}$. Since $K\widetilde{L}=W_1$, such a $g$ also stabilizes each of $G_n(L_i)$, for $0 \leq i \leq d$. Hence the required inclusion of subgroups. 
		\end{proof}

		From the Definition \ref{D:adj_2}, we can check that $G_n$ and $G_{n+1}$ are adjacent maps with respect to $\hat{\epsilon}$. Indeed, condition \ref{D:adj_2}(i) follows since for any $L$, $G_n(L) \subset G_{n+1}(L)$. Also if $L \subsetneq M$ are two $R$-lattices such that $\hat\epsilon(L) < \hat\epsilon(M)$, then $G_{n+1}(L) = L + \varpi^{-n-1-\hat\epsilon(L)}\widetilde{L} \subset M + \varpi^{-n-\hat\epsilon(M)}\widetilde{L}$, hence condition \ref{D:adj_2}(ii) is also satisfied. Proposition \ref{P:adj_2} shows then that the geometric realizations of $\langle G_n\rangle $ and $\langle G_{n+1}\rangle$ are homotopic to each other. But we will need a homotopy between the realizations of $[G_n]$ and $[G_{n+1}]$, which can be obtained by following the proof of the same result as follows.

		\underline{\textit{Claim 3.}} $|[G_n]|$ and $|[G_{n+1}]|$ are homotopic to each other.  
		
		\begin{proof}
			Let $(\BB(W)_{\sigma,0})_\bullet$ be the simplicial subcomplex of $\BB(W)_{\sigma}$, whose set of vertices is given by $\hat\epsilon^{-1}(0)$ and a chain $(L_0 \subset \ldots \subset L_d)$ is defined to be a $d$-simplex, if $\gamma=\{\langle L_0\rangle \ldots , \langle L_d\rangle\} $ is a $d$-simplex in $\cB_\bullet(W)_\sigma$ and so in particular $\exists g \in \Gamma_\gamma \setminus \{1\}$ such that $g_{|W_1}=1_{W_1}$. Similar to Remark \ref{R:aug_map}, we obtain the natural isomorphism $(\BB(W)_{\sigma,0})_\bullet \simeq \cB_\bullet(W)_{\sigma}$. Denote the composition of $[G_n]$ resp. $[G_{n+1}]$ with this isomorphism by $[G_n]_0$ resp. $[G_{n+1}]_0$. For a $d$-simplex $\gamma=(L_0 \subset \ldots \subset L_d)$ in $(\BB(W)_{\sigma,0})_\bullet$, we have $\hat\epsilon(L_d)=0 <1=\hat\epsilon(\varpi^{-1}L_0)$. 
		It follows that for any $0 \leq i \leq d$,
		\[\tilde{\gamma} = (G_n(L_0) \subset \ldots  G_n(L_i) \subset G_{n+1}(L_i) \subset \ldots G_{n+1}(L_d))) \]
		is a $(d+1)$-simplex in $(\BB(W)_{\sigma,0})_\bullet$, because $G_n$ and $G_{n+1}$ are adjacent, we have $G_{n+1}(L_d) \subsetneq G_{n}(\varpi^{-1}L_0)$ and also by Claim 2, if $g \in \Gamma_{\gamma} \setminus \{1\}$, then $g \in \Gamma_{G_i(L_j)} \setminus \{1\}$ for $i=n,n+1$ and for all $0 \leq j \leq d$, and hence $g \in \Gamma_{\tilde{\gamma}} \setminus \{1\}$. 
Consequently $[G_n]_0$ and $[G_{n+1}]_0$ are adjacent in the sense of Definition~\ref{D:adj_1}, and we conclude from Proposition \ref{P:adj_1} that the geometric realizations $|[G_n]_0|$ and $|[G_{n+1}]_0|$ are homotopic, and hence that $|[G_n]|$ and $|[G_{n+1}]|$ are homotopic.
%
		\end{proof}

		Now the rest of the proof of the theorem is exactly as in \cite[Thm.4.1]{grayson}. For any $R$-lattice $L \subset W$ of maximal rank, we have \[G_n(L):= \begin{cases}
			L;\  n \ll 0 \\
			G_n \beta\alpha(L); \ n\gg 0
		\end{cases}
		.\] We can now finish the proof using Whitehead's theorem(Corollary \ref{C:cor_equi_white}) in the standard, non-equivariant form; cf.~Remark~\ref{R:non-equi-whitehead}. Let $f:Z \rightarrow |\cB_{\bullet}(W)_{\sigma}|$ be a continuous map from a compact topological space $Z$. Then its image is in the span of finitely many vertices of $\cB_{\bullet}(W)_{\sigma}$. Hence by above, we have \[
		[G_n] \circ f = \begin{cases}
			f;\ n\ll 0\\
			[G_n] \circ [ \beta ] \circ [\alpha] \circ f; \ n\gg0
		\end{cases}
		.\] 
		Consequently $|f|$ is homotopic to $|[G_n]| \circ |[ \beta ]| \circ |[\alpha]| \circ |f|$ by claim 3. Since $|[\alpha]| \circ |f|$ is null-homotopic(since the Bruhat-Tits building is contractible see \cite[Thm.2.1]{grayson} for a proof), $|[G_n]| \circ |[ \beta ]| \circ |[\alpha]| \circ |f|$ is also null-homotopic, and we can conclude that $|f|$ is null-homotopic. By Whitehead's theorem, $|\cB_{\bullet}(W)_{\sigma}|$ is contractible. 
	\end{proof}

	\subsection{The $\Gamma_P$-equivariant homotopy $\cB_\bullet(W)^{\Gamma\dashun} \rightarrow \mathcal{T}_\bullet(W)$}  
	
	Define the $\Gamma$-unstable subcomplex of $\cB_\bullet(W)$, denoted $\cB_\bullet(W)^{\Gamma\dashun}$, as follows. A simplex $\gamma=\{\langle L_0\rangle,\ldots,\langle L_d\rangle\}$ is in $\cB_\bullet(W)^{\Gamma\dashun}$ if and only if $\Gamma_\gamma \neq \{1\}$.  
	
	Recall that $\Gamma_P=\Aut_A(P)$. Analogous to \cite[Cor.4.2]{grayson}, we would like to use the previous Theorem \ref{T:un_contr} to conclude a $\Gamma_P$-equivariant homotopy equivalence between $\cB_\bullet(W)^{\Gamma\dashun}$ and $\mathcal{T}_\bullet(W)$. The strategy to construct the sought after $\Gamma_P$-equivariant homotopy equivalence is to use Theorem \ref{T:lem1.9}. In the notation of that theorem, we apply it as follows. Set $X:=\cB_\bullet(W)^{\Gamma\dashun}$ and $\mathcal{T}:=\mathcal{T}_\bullet(W)$. For $\sigma \in \mathcal{T}$, define $X_\sigma:=\cB_\bullet(W)_{\sigma}$. 
	
	In this setting \ref{T:lem1.9}(i) is obvious and \ref{T:lem1.9}(iii) follows from Theorem \ref{T:un_contr}. Also \ref{T:lem1.9}(ii) follows from the following standard fact from group representation theory over positive characteristic fields. 
	
	\begin{lem}(cf. \cite[Lem.38]{kondo})\label{L:basic_lemma}
		Let $G$ be a finite $p$-group acting on a finite dimensional $K$-vector space $V$. Then $V^G \neq \{0\}$. 
	\end{lem}
	
	\begin{proof}
		We induct on the size of $G$. Suppose $G=\langle g \rangle$ is a cyclic group generated by $g$. Then $g^{p^n}=1$ for some $n \in \ZZ$. Consequently, the characteristic polynomial of $g$ is of the form $X^{p^n}-1$, so the only eigenvalue for the action of $g$ on $V$ is $1$. Hence $V^{\langle g \rangle} \neq \{0\}$.
		
		Now let $G$ be an arbitrary $p$-group. It is well known fact that the centralizer $C_G \neq \{1\}$, so pick a $g \in C_G$. Then $\langle g \rangle \subset G$ is a normal subgroup and $G/\langle g\rangle$ acts on $V^{\langle g \rangle}$. By induction hypothesis, we have $V^G=(V^{\langle g\rangle})^{G/\langle g\rangle} \neq \{0\}$.  
	\end{proof}
	
	Now \ref{T:lem1.9}(ii) follows in our setting: Given any simplex $\gamma$ in $\cB_{\bullet}(W)^{\Gamma\dashun}$ we have $\Gamma_{\gamma} \neq \{1\}$. From Corollary~\ref{C:prin_stab} and Lemma~\ref{L:basic_lemma}, we deduce $\{0\} \subsetneq W^{\Gamma_\gamma} \subsetneq W$. Hence $\{0\} \subsetneq W^{\Gamma_\gamma} \subsetneq W$ determines a vertex $\sigma \in \mathcal{T}_\bullet(W)$ and $\gamma$ is a simplex of $\cB_\bullet(W)_{\sigma}$.

	Finally one needs to show that given a simplex $\gamma \in \cB_\bullet(W)^{\Gamma\dashun}$, the poset \[\mathcal{T}_\bullet(W)_\gamma:=\{\sigma \in \mathcal{T}_\bullet(W) | \ \gamma \text{ is a simplex of }\cB_\bullet(W)_{\sigma}\}\] is $(\Gamma_P)_\gamma$-contractible. Since $\Gamma_\gamma \trianglelefteq (\Gamma_P)_\gamma$ is a normal subgroup, we will achieve this by considering the $(\Gamma_P)_\gamma$-map of posets 
	\[\mathcal{T}_\bullet(W)_{\gamma} \xrightarrow{\phi} \mathcal{T}_\bullet(W)_\gamma\] defined by \[\sigma=(0 \subsetneq W' \subsetneq W) \mapsto (0 \subsetneq W'+W^{\Gamma_\gamma} \subsetneq W).\]This is a well defined map since, $W'+W^{\Gamma_\gamma}$ is a proper subspace of $W$ as any $g \in \Gamma_\gamma$ stabilizes $W'+W^{\Gamma_\gamma}$ and hence it must not be the whole vector space $W$. 
	Consequently, since $\phi(\sigma) \geq \sigma$ for all $\sigma \in \mathcal{T}_\bullet(W)_\gamma$, by Corollary \ref{C:quillen1.3_2} the poset $\mathcal{T}_\bullet(W)_\gamma$ is $(\Gamma_P)_\gamma$-homotopy equivalent to the poset $\im(\phi)$. But $\im(\phi)$ has a minimal element, namely $W^{\Gamma_\gamma}$, which is also $(\Gamma_P)_\gamma$-invariant. Hence by Corollary \ref{C:quillen1.3_1} $\im(\phi)$ is $(\Gamma_P)_\gamma$-contractible, or $\mathcal{T}_\bullet(W)_\gamma$ is also $(\Gamma_P)_\gamma$-contractible. 
	
	To summarize the discussion in this section, we have
	
	\begin{thm}\label{T:main_thm_2}
		There is a natural $\Gamma_P$-equivariant homotopy equivalence between $|\cB_\bullet(W)^{\Gamma\dashun}|$ and $|\mathcal{T}_\bullet(W)|$. 
	\end{thm}

	\begin{rem}\label{R:main_thm_2}
	    \begin{itemize}
	    
		\item[(i)] Since $\cB_\bullet(W)^{\Gamma\dashun}$ and $\mathcal{T}_\bullet(W)$ are simplicial complexes, it follows that their corresponding (reduced)simplicial homology groups of the geometric realizations are also naturally isomorphic. In particular, for all $i\geq0$ there is an isomorphism $\tilde{\h}_i(\cB_\bullet(W)^{\Gamma\dashun},\ZZ) \simeq \tilde{\h}_i(\mathcal{T}_\bullet(W),\ZZ)$ that is also $\Gamma_P$ equivariant (the former is contained in any standard text on algebraic topology, eg. \cite[Chap.2, \S19]{Mun84}, and the latter follows from the $\Gamma_P$-equivariance of the homotopy equivalence). 
		\item[(ii)] It is worth highlighting that when $r=2$ the homotopy equivalence in Theorem \ref{T:main_thm_2} is not a $\Gamma_P$ homotopy equivalence(i.e., there do not exist a $\Gamma_P$-equivariant homotopy inverse $|\mathcal{T}_\bullet(W)| \rightarrow |\cB_\bullet(W)^{\Gamma\dashun}|$. See Definition \ref{D:homotopy_equi}(ii)) and we expect the same to hold true in higher ranks as well. Indeed to see this fact in $r=2$ case, recall that $\mathcal{T}_\bullet(W) = \mathbb{P}(W)$ and note by (i) that an inverse $\Gamma_P$-equivariant homotopy inverse $|\mathbb{P}(W)| \rightarrow |\cB_\bullet(W)^{\Gamma\dashun}|$ would result in a $\Gamma_P$ equivariant map $g:\C_0(\mathbb{P}(W),\ZZ) \rightarrow \C_0(\cB_\bullet(W)^{\Gamma\dashun},\ZZ)$. In particular, for each $c \in \mathbb{P}(W)$, $g([c]) = \underset{x \in \cB_0(W)^{\Gamma\dashun}}{\sum} n_x[x]$, where $n_x \in \ZZ$ are almost all zero. The stabilizer $(\Gamma_P)_c$ has infinite cardinality(see \cite[Chap.2,\S2.5]{serre}) whereas the (finitely many) vertices occuring in the sum $g([c])$ have finite $\Gamma_P$ stabilizers. Hence the $\Gamma_P$-equivariance of $g$ means that $g\equiv 0$. But this is a contradiction since $\C_\bullet(\mathbb{P}(W),\ZZ)$ is not contractible.
		\end{itemize}
 	\end{rem}
	\section{Relating the various (un)stable complexes}
	
	In this section, we will be interested in understanding the compatibility of the isomorphisms in Theorem \ref{T:main_thm_2} when $\Gamma$ runs over $\Gamma_P(I)$ for $I$ any proper non-zero ideal of $A$. In particular, we will show a compatible family of isomorphisms between the homology of the various $\Gamma_P(I)$-\textit{stable complexes}(defined below) with the Steinberg module. Let $\Gamma:=\Gamma_P(I)$ as before.  First we will explicitly describe a simplicial homological complex for the building $\cB_\bullet(W)$ which consequently also defines it for $\cB_\bullet(W)^{\Gamma\dashun}$. Let us fix a $K$-basis $\{e_1,\ldots,e_r\}$ of $W$ and by $type(x)$ for a vertex $x \in \cB_0(W)$ we mean the type with respect to this fixed basis; see~Remark~\ref{R:type_i_vertex}. 
	
\begin{defi}\label{D:Complex_BT}
We denote by $(\C_i(\cB_\bullet(W),\ZZ),\partial_i)$ the homological complex which is defined as follows:
    \begin{itemize}
        \item[(i)] $\C_i(\cB_\bullet(W),\ZZ) = 0$ for all $i>r$ and $i<-1$;
        \item[(ii)] $\C_{-1}(\cB_\bullet(W),\ZZ) = \ZZ$
        \item[(iii)] For $0 \leq i \leq r$, \[ \C_i(\cB_\bullet(W),\ZZ) := \underset{\substack{\{x_0,\ldots,x_i\} \in \cB_i(W) \\ type(x_0) < \ldots < type(x_i)}}{\bigoplus} \ZZ [x_0,\ldots,x_i],\] where for a simplex $s=\{x_0,\ldots,x_i\} \in \cB_i(W)$, we denote by $[s]=[x_0,\ldots,x_i]$ the corresponding basis symbol in $\C_i(\cB_\bullet(W),\ZZ)$.  
        \item[(iv)] For $1 \leq i \leq r$, the differentials are given by
            \[\partial_i:\C_i(\cB_\bullet(W),\ZZ) \rightarrow \C_{i-1}(\cB_\bullet(W),\ZZ)\] sending \[[x_0,\ldots,x_i] \mapsto \underset{0 \leq j \leq i}{\sum} (-1)^j[x_0,\ldots,\widehat{x_j},\ldots,x_i].\]$\partial_0$ sends $[x]$ to $1$, for a vertex $x \in \cB_0(W)$.
    \end{itemize}

\end{defi}

\begin{rem}
    \begin{itemize}
    	\item[(i)] The homology of the complex $(\C_i(\cB_\bullet(W),\ZZ),\partial_i)$ is the reduced homology of the space $|\cB_\bullet(W)|$, so we will denote it by $\tilde{\h}_{\bullet}(\cB_\bullet(W),\ZZ)$.  
        \item[(ii)] As already alluded to before, $|\cB_\bullet(W)|$ is contractible, that is, $\tilde{\h}_i(\cB_\bullet(W),\ZZ) = 0$ for all $i \neq 0$.
        \item[(iii)] It is clear that the complex $\C_\bullet(\cB_\bullet(W),\ZZ)$ is equipped with the obvious $\GL(W)$-action making it a complex of $\ZZ[\GL(W)]$-modules.
    \end{itemize}
\end{rem}

Denote by $\C_\bullet(\cB_\bullet(W)^{\Gamma\dashun},\ZZ) \subset \C_\bullet(\cB_\bullet(W),\ZZ)$ the subcomplex involving only the $\Gamma$ unstable simplices and then its homology will be denoted by $\tilde{\h}_{\bullet}(\cB_\bullet(W)^{\Gamma\dashun},\ZZ)$.   
	
	\begin{defi}\label{D:stable_complex}
 Define \[C_\bullet(\mathcal{B}_\bullet(W),\ZZ)^{\Gamma\dashst}:= \coker(C_\bullet(\mathcal{B}_\bullet(W)^{\Gamma\dashun},\ZZ) \hookrightarrow C_\bullet(\mathcal{B}_\bullet(W),\ZZ)).\]The homology of the complex $C_\bullet(\mathcal{B}_\bullet(W),\ZZ)^{\Gamma\dashst}$ will be denoted by $\h_\bullet(\mathcal{B}_\bullet(W)^{\Gamma\dashst},\ZZ)$. 
		
	\end{defi}
	

	\begin{rem}\label{R:stable_complexes1}
	The set of stable simplices does not form a simplicial complex. Therefore, a priori it is not clear whether $C_\bullet(\mathcal{B}_\bullet(W),\ZZ)^{\Gamma\dashst}$ can be defined as the complex associated to a subset of simplices of $\cB_\bullet(W)$. To clarify this, let $\cB_d(W)^{\Gamma\dashst}:=\cB_d(W)\setminus \cB_d(W)^{\Gamma\dashun}$ for $d \geq 0$, so that $\cB_d(W)^{\Gamma\dashst}$ consists of those $d$-simplices with trivial $\Gamma$-stabilizer. Let $C_d(\cB_\bullet(W)^{\Gamma\dashst},\ZZ)$ be the free $\ZZ$-module on generators $[s]$ for $s\in \cB_d(W)^{\Gamma\dashst}$, and define the differential $\partial_d:C_d(\cB_\bullet(W)^{\Gamma\dashst},\ZZ) \rightarrow C_{d-1}(\cB_\bullet(W)^{\Gamma\dashst},\ZZ)$ by 
			\[[v_0,\ldots,v_d]_{\st} \mapsto \sum^{d}_{i=0} (-1)^i [v_0,\ldots, \widehat{v_i},\ldots,v_d]_{\st},\] 
where for a simplex $s=\{t_0,\ldots,t_j\}$ the symbol $[s]_{\st} = 0$ if $s$ is $\Gamma$-unstable, and otherwise $[s]_{\st}$ is the symbol $[s]$ defined above. Then as the reader may easily verify, this defines a complex $C_\bullet(\cB_\bullet(W)^{\Gamma\dashst},\ZZ)$ and one has a canonical isomorphism $ C_\bullet(\mathcal{B}_\bullet(W),\ZZ)^{\Gamma\dashst}\to C_\bullet(\cB_\bullet(W)^{\Gamma\dashst},\ZZ)$.
	\end{rem}
	\smallskip
	
%
			Since $\Gamma\trianglelefteq \Gamma_P$ is a normal subgroup, the action of $\Gamma_P$ on $\cB_\bullet(W)$ induces a natural action of $\Gamma_P$ on $\cB_\bullet(W)^{\Gamma\dashst}$. Consequently we have a short exact sequence of $\Gamma_P$-modules \[0 \rightarrow C_\bullet(\cB_\bullet(W)^{\Gamma\dashun},\ZZ) \rightarrow C_\bullet(\cB_\bullet(W),\ZZ) \rightarrow C_\bullet(\cB_\bullet(W),\ZZ)^{\Gamma\dashst} \rightarrow 0,\]
			giving rise to a long exact sequence of $\Gamma_P$-modules 
			\begin{equation}\label{eq:les-st-ust}\ldots \!\rightarrow\! \tilde{\h}_i(\cB_\bullet(W)^{\Gamma\dashun},\ZZ) \!\rightarrow \!\tilde{\h}_i(\cB_\bullet(W),\ZZ)\! \rightarrow\! \h_i(\cB_\bullet(W)^{\Gamma\dashst},\ZZ)\! \rightarrow \!\tilde{\h}_{i-1}(\cB_\bullet(W)^{\Gamma\dashun},\ZZ) \!\rightarrow\! \ldots 
			\end{equation}
	
	Because the Bruhat-Tits building $\cB_\bullet(W)$ is contractible, the above long exact sequence together with Remark~\ref{R:main_thm_2} yields isomorphisms of $\Gamma_P$-modules 
	\begin{equation}\label{eq:iso-st-ust}
	\h_{i-1}(\cB_\bullet(W)^{\Gamma\dashst},\ZZ) \stackrel{\eqref{eq:les-st-ust}}\simeq \tilde{\h}_{i-2}(\cB_\bullet(W)^{\Gamma\dashun},\ZZ) \stackrel{\mathrm{Thm.}\ref{T:main_thm_2}}\simeq \tilde{\h}_{i-2}(\mathcal{T}_\bullet(W),\ZZ) 
	\end{equation}
	for all $i\ge2$, and the vanishing $\h_{0}(\cB_\bullet(W)^{\Gamma\dashst},\ZZ) =0$. Moreover as recalled in Theorem~\ref{solomon_quillen}, one also has $\tilde{\h}_{i}(\mathcal{T}_\bullet(W),\ZZ)=0 $ unless $i=r-2$ and $\tilde{\h}_{r-2}(\mathcal{T}_\bullet(W),\ZZ)=\St(W)$. Hence the augmented complex 
		\begin{equation}\label{eq:Augm-Cpx}0\to \h_{r-1}(\cB_\bullet(W)^{\Gamma\dashst},\ZZ)\to  C_\bullet(\cB_\bullet(W),\ZZ)^{\Gamma\dashst} 
	\end{equation}
	 is acyclic. This motivates the following definition:
	\begin{defi}
		The Steinberg module for level $I$ is defined as
		\[\St^{I}(P):=\h_{r-1}(\cB_\bullet(W)^{\Gamma\dashst},\ZZ),\]
		and the isomorphism $St^{I}(P) \simeq \St(W)$ of $\Gamma_P$-modules from \eqref{eq:iso-st-ust} is denoted by $\GQ_I$.
	\end{defi}	
%
	
\begin{lem}\label{lem:finitelyManyOrbits}
Each set $\cB_d(W)^{\Gamma\dashst}$ is a finite union of free $\Gamma$-orbits. 
\end{lem}
\begin{proof}
We rely on the proof of \cite[Prop.13.8]{Bux_2013}: Fix a $K$-basis $\{e_i\}^{r}_{i=1}$ of $W$ and consider the setting $S:=\{\infty\},\mathcal{G}:=(\GL_r)_K$ in the notation of op.cit., so that $\mathcal{G}(K)=\GL(W)$ w.r.t. $\{e_i\}^{r}_{i=1}$ and for $\GL^{0}_r(A):=\im(\GL_r(A) \rightarrow \Aut(\cB_\bullet(W)))$, the series \[\underset{C \in \GL_r(A)  \backslash \cB_{r-1}(W)}{\sum} \frac{1}{|\GL^0_r(A)_C|}\]converges by \cite[Prop.13.8]{Bux_2013}. The kernel of $\GL_r(A) \rightarrow \Aut(\cB_\bullet(W))$ is a finite group by Proposition \ref{P:stab-finite}, hence the series \[\underset{C \in \GL_r(A)  \backslash \cB_{r-1}(W)}{\sum} \frac{1}{|\GL_r(A)_C|}\]also converges. In particular, given $M \in \NN$ there are finitely many $\GL_r(A)$ orbit of simplices with the order of the associated $\GL_r(A)$ stabilizer bounded above by $M$. Since $\Gamma$ and $\GL_r(A)$ are commensurable from Lemma \ref{L:commensurable}, we can conclude that there are at most finitely $\Gamma$ orbits of $\Gamma$-stable simplices. 
\end{proof}
	
	It follows from Lemma~\ref{lem:finitelyManyOrbits} and Remark~\ref{R:stable_complexes1} that $C_\bullet(\cB_\bullet(W),\ZZ)^{\Gamma\dashst} $ is a complex of free and finitely generated $\ZZ[\Gamma]$-modules, and from the acyclicity of \eqref{eq:Augm-Cpx} we deduce:
	
		
		\begin{cor}
		The $\Gamma_P$-module $\St(W) $ is projective and finitely~generated  as a $\ZZ[\Gamma]$-module.
		\end{cor}
		
%
%

	Our next result concerns the compatibility of the maps $\GQ_I$ as $I$ runs through non-zero proper ideals of $A$:
	\begin{thm}\label{thm:Thm2}
	Let $J \subset I \subset A$ be a non-zero proper ideal and set $\Gamma' := \Gamma_P(J)$, so that $\Gamma' \trianglelefteq \Gamma \trianglelefteq \Gamma_P$. Then there is a canonical $\Gamma_P$-equivariant homomorphism of abelian groups $i_{J,I} : \St^{J}(P) \rightarrow \St^{I}(P)$ such that the following diagram commutes 
	\begin{equation}\label{E:3}
		\begin{tikzcd}[row sep=tiny]
			\St^{J}(P) \arrow[dd,"i_{J,I}"]
			\arrow[dr,"\GQ_{J}"{near start}
			] &  \\
			& \St(W)   \\
			\St^{I}(P) \arrow[ur,"\GQ_I" '{near start}] & 
		\end{tikzcd}
	\end{equation}
	\end{thm}
\begin{proof}	
	
	Note that any simplex that is $\Gamma'$-unstable is also $\Gamma$-unstable. Hence we have an embedding of complexes, $C_\bullet(\cB_\bullet(W)^{\Gamma'\dashun},\ZZ) \hookrightarrow C_\bullet(\cB_\bullet(W)^{\Gamma\dashun},\ZZ)$ giving rise to the quotient of the corresponding stable complexes  $C_\bullet(\cB_\bullet(W),\ZZ)^{\Gamma'\dashst} \twoheadrightarrow C_\bullet(\cB_\bullet(W),\ZZ)^{\Gamma\dashst}$. This yields natural $\Gamma_P$-equivariant maps $\pi_{J,I}:\tilde{\h}_{r-2}(\cB_\bullet(W)^{\Gamma'\dashun},\ZZ) \rightarrow \tilde{\h}_{r-2}(\cB_\bullet(W)^{\Gamma\dashun},\ZZ)$ as well as $i_{J,I}:\St^{J}(P) \rightarrow \St^I(P)$ such that the following diagram of $\Gamma_P$-modules commutes 
	\begin{equation}\label{E:2}
		\begin{tikzcd}
			\St^{J}(P) \arrow[r] \arrow[d,"i_{J,I}"] 
			& \tilde{\h}_{r-2}(\cB_\bullet(W)^{\Gamma'\dashun},\ZZ)\arrow[d,"\pi_{J,I}"]\\
			\St^{I}(P) \arrow[r]  
			& \tilde{\h}_{r-2}(\cB_\bullet(W)^{\Gamma\dashun},\ZZ)
		\end{tikzcd}
	\end{equation}	
%
Therefore it will suffice 
	to show the commutativity of
	\begin{equation}\label{E:4}
		\begin{tikzcd}[row sep=tiny]
			\tilde{\h}_{r-2}(\cB_\bullet(W)^{\Gamma'\dashun},\ZZ)\arrow[dd,"\pi_{J,I}"]\arrow[dr,"\alpha_{J}"{near start}] &  \\
			& \St(W)   \\
			\h_{r-2}(\cB_\bullet(W)^{\Gamma\dashun},\ZZ) \arrow[ur,"\alpha_I" '{near start}] & 
		\end{tikzcd}
	\end{equation}

	We briefly recall how $\alpha_J$ and $\alpha_I$ arises out of the quasi isomorphism in \ref{T:main_thm_2}, which in turn is invoked from the Theorem \ref{T:lem1.9}. For $\ast \in \{\Gamma',\Gamma\}$, let $\cB_{\bullet}(W)_{1}^{\ast\dashun}$ denote the first barycentric subdivision of $\cB_{\bullet}(W)^{\ast\dashun}$ and $Z_\ast \subset \cB_{\bullet}(W)_{1}^{\ast\dashun} \times \mathcal{T}_\bullet(W)$ be the closed subposet such that $(\gamma,\sigma) \in Z_\ast$ if and only if $\gamma \in \cB_\bullet(W)_\sigma$(note that here the $\cB_\bullet(W)_\sigma$ are defined with respect to the corresponding $\ast$. The reader is referred to Subsection 4.1 for details on the definition of the subcomplex $\cB_\bullet(W)_\sigma$). Then there is a correspondence of the homological complexes 
	
	\begin{equation}\label{fig:6}
		\begin{tikzcd}
			&  & C_\bullet(Z_\ast,\ZZ)\arrow[dl,"p_{1,\ast}" '{auto}] \arrow[dr,"p_{2,\ast}"{auto}] &  \\
			C_{\bullet}(\cB_\bullet(W)^{\ast\dashun},\ZZ) \arrow[r,"\sim"] & C_{\bullet}(\cB_{\bullet}(W)_1^{\ast\dashun},\ZZ) & & C_\bullet(\mathcal{T}_\bullet(W),\ZZ)
		\end{tikzcd}
	\end{equation}where $p_{1,\ast}$ resp. $p_{2,\ast}$ is the first resp. second projection. It turns out, as argued in the proof of \ref{T:main_thm_2}, that $p_{1,\ast}$ is a $\Gamma_P$-homotopy equivalence and consequently this induces $\alpha_J$ resp. $\alpha_I$ for $\ast=\Gamma'$ resp. $\Gamma$. Consequently it is clear that we need to show only that the following two diagrams commute, which is clear from construction 
	\[
	\begin{tikzcd}
		C_\bullet(\cB_{\bullet}(W)_1^{\Gamma'\dashun},\ZZ) \arrow[d] 
		& C_\bullet(Z_{\Gamma'},\ZZ) \arrow[l,"p_{1,\Gamma'}"{sloped,auto}]\arrow[d]\\
		C_\bullet(\cB_{\bullet}(W)_1^{\Gamma\dashun},\ZZ)  
		& C_\bullet(Z_{\Gamma},\ZZ) \arrow[l,"p_{1,\Gamma}"{sloped,auto}]
	\end{tikzcd} \text{and}
	\begin{tikzcd}[row sep=tiny]
		C_\bullet(Z_{\Gamma'},\ZZ)\arrow[dd]\arrow[dr,"p_{2,\Gamma'}"{near start}] &  \\
		& C_\bullet(\mathcal{T}_\bullet(W),\ZZ)   \\
		C_\bullet(Z_\Gamma,\ZZ) \arrow[ur,"p_{2,\Gamma}" '{near start}] & 
	\end{tikzcd}
	\]Here the right vertical map in the first diagram(or the vertical map in the second) is the natural embedding, similar to that of the left vertical embedding. Consequently we get the commutativity of diagram \eqref{E:3}.
\end{proof}

\begin{rem}
	\begin{itemize}
		\item[(i)] Choosing $J$ to be in the descending chain of ideals $\{I^n\}_{n \geq 1}$, we get a sequence of simplicial complexes, $\ldots \subset \cB_\bullet(W)^{\Gamma_P(I^{n+1})\dashun} \subset \cB_\bullet(W)^{\Gamma_P(I^n)\dashun} \subset \ldots$, (and also a corresponding increasing sequence $\ldots \subset \cB_\bullet(W)^{\Gamma_P(I^n)\dashst} \subset \cB_\bullet(W)^{\Gamma_P(I^{n+1})\dashst} \subset \ldots$). The results of this section shows that the corresponding (top dimensional) simplicial homology is compatibly $\Gamma_P$-equivariantly isomorphic to the Steinberg module~$\St(W)$. 
	    \item[(ii)] A different family of subcomplexes of $\cB_\bullet(W)$ homotopy equivalent to the Tits building is also considered in \cite[\S11.2,11.3]{kondo} using the theory of Harder-Narasimhan polygons, and following more closely~\cite{grayson}.  
	\end{itemize} 
\end{rem}

	\printbibliography
\end{document}